\documentclass{article}
\usepackage{geometry}                
\usepackage{graphicx}
\usepackage{amssymb}
\usepackage{epstopdf}
\usepackage{theorem}
\usepackage{amsmath,amssymb}  
\usepackage{bm}  
\usepackage{color}
\usepackage{xypic}
\usepackage[applemac]{inputenc}
\usepackage[pdftex,bookmarks,colorlinks,breaklinks]{hyperref}  
\definecolor{dullmagenta}{rgb}{0.4,0,0.4}   
\definecolor{darkblue}{rgb}{0,0,0.4}
\hypersetup{linkcolor=black,citecolor=black,filecolor=black,urlcolor=black} 

\usepackage{memhfixc} 

\newtheorem{theorem}{Theorem}[section]
\newtheorem{proposition}[theorem]{Proposition}
\newtheorem{lemma}[theorem]{Lemma}
\newtheorem{cor}[theorem]{Corollary}

{\theorembodyfont{\normalfont\rmfamily}
\newtheorem{definition}[theorem]{Definition}

\newtheorem{remark}[theorem]{Remark}

\newtheorem{exemple}[theorem]{Example}

} \makeatletter 

\makeatletter
\def\operatorname#1{\mathop{\operator@font #1}\nolimits}%
\makeatother
\newcommand{\R}{\mathbb{R}}
\newcommand{\C}{\mathbb{C}}
\newcommand{\Q}{\mathbb{Q}}
\newcommand{\Z}{\mathbb{Z}}

\newcommand{\half}{{\frac{1}{2}}}
\newcommand{\thalf}{{\tfrac{1}{2}}}

\newcommand{\End}{\operatorname{End}}

\newcommand{\im}{\operatorname{Im}}
\newcommand{\Spec}{\operatorname{Spec}}
\newcommand{\Ham}{\mathcal{H}_{\textrm{MB}}}
\newcommand{\Hstd}{\mathcal{H}_{\textrm{std}}}
\newcommand{\Hs}{\mathcal{H}}
\newcommand{\J}{\mathcal{J}}
\newcommand{\Lie}{\mathcal{L}}
\newcommand{\M}{\operatorname{Max}}
\newcommand{\m}{\operatorname{min}}
\newcommand{\Per}{\mathcal{P}}
\newcommand{\A}{\mathcal{A}}
\newcommand{\Mod}{\mathcal{M}}

\newcommand{\w}{\omega}
\newcommand{\abs}[1]{\left| #1 \right|}
\newcommand{\norm}[1]{\left\| #1 \right\|}

\newenvironment{proof}[1][{}]{{ \textsc{Proof{#1}:~}}}{{\hfill$\square$\\ }}

\makeatletter
\newcommand*{\rom}[1]{\expandafter\@slowromancap\romannumeral #1@}
\makeatother

\def\adots{\mathinner{\mkern2mu\raise 1pt\hbox{.}\mkern 3mu\raise
4pt\hbox{.}\mkern2mu\raise 8pt\hbox{{.}}}}

\title{The positive equivariant symplectic homology as an  invariant for some contact manifolds}
\author{Jean Gutt}
	
\date{ }                                           

\begin{document}
\maketitle

\begin{abstract}
We show that positive $S^1$-equivariant symplectic homology is a contact invariant for a subclass of contact manifolds which are boundaries of Liouville domains.
In nice cases, when the set of Conley-Zehnder indices of all good periodic Reeb orbits on the boundary of the Liouville domain is lacunary, the positive $S^1$-equivariant symplectic homology can be computed;  it is generated by those orbits.
We prove a ``Viterbo functoriality'' property: when one  Liouville domain  is embedded into an other one, there is a morphism (reversing arrows)
between their positive $S^1$-equivariant symplectic homologies and morphisms compose nicely.

These properties allow us to give a proof of Ustilovsky's result  on the number of non isomorphic contact structures on the spheres $S^{4m+1}$.
They also give a new proof of a Theorem by Ekeland and Lasry on the minimal number of periodic Reeb orbits on some hypersurfaces in $\R^{2n}$.
We extend this result to some hypersurfaces in some negative line bundles.
\end{abstract}

\section{Introduction}
A \emph{contact structure} on a manifold $M$ of dimension $2n-1$ is a hyperplane field $\xi$ which is maximally non integrable;
i.e. writing locally $\xi = \ker \alpha$, one has 
$\alpha \wedge (d\alpha)^{n-1} \neq 0$ everywhere. A first natural question is to determine
{\textbf{ how many non-diffeomorphic contact structures  there are on a given manifold and in particular on the sphere $S^{2n-1}$}}.
In this paper we consider only coorientable contact structures, i.e. such that $\alpha$ can be  globally defined; $\alpha$ is then called a {\emph{contact form}};
this form is not unique; for any function $f:M\rightarrow\R$, the $1$-form $e^f\alpha$ defines the same contact structure.
The \emph{Reeb vector field} $R_{\alpha}$ associated to a contact form $\alpha$ is the unique vector field on $M$ such that
$\iota(R_{\alpha})d\alpha = 0$ and $\alpha (R_{\alpha}) = 1$.
Since this vector field does not vanish anywhere, there are no fixed points of its flow.
Periodic orbits are thus the most noticeable objects in the flow.
If $(M,\xi)$ is a compact contact manifold, can one say something about \textbf{the minimal number of geometrically distinct periodic Reeb orbits for any contact form $\alpha$
(eventually in a subclass) defining the contact structure $\xi$?}
A way to tackle those two questions is to find an invariant of the contact structure built out of the periodic Reeb orbits.
To build such an invariant  is the aim of contact homology.
At the time of this writing,  contact homology is still in development and encounters ``transversality'' problems.
Instead we consider positive $S^1$-equivariant symplectic homology which is built from periodic orbits of Hamiltonian vector fields in a symplectic manifold whose boundary is the given contact manifold.
In this spirit,  Bourgeois and Oancea, in \cite{bo}, relate, in the case where it can be defined, the linearised contact homology of the boundary to the positive $S^1$-equivariant symplectic homology of the symplectic manifold.

The goal of this paper is to prove that positive $S^1$-equivariant symplectic homology is a contact invariant for a subclass of contact manifolds and that this allows
to obtain results about the two questions mentioned earlier.

In Section \ref{chap:background}, we recall the definition of positive $S^1$-equivariant symplectic homology, first describing symplectic homology $SH$, 
positive symplectic homology $SH^+$, and $S^1$-equivariant symplectic homology $SH^{S^1}$.
We show in Section \ref{chap:SHS1orbits}  that, in nice cases,  generators of the positive $S^1$-equivariant symplectic homology $SH^{S^1,+}$
 are given by good periodic Reeb orbits. 
This relies heavily on earlier results from Bourgeois and Oancea \cite{BOduke} and recent results from Zhao \cite{Zhao}. Precisely, we prove 
\begin{theorem}\label{thm:computing}
	Let $(W,\lambda)$ be a Liouville domain.
	Assume there exists a contact form $\alpha$ on the boundary $\partial W$ such that the set of Conley-Zehnder indices of all good periodic Reeb orbits is lacunary\footnote{A set of integer numbers is
	\emph{lacunary} if it does not contain two consecutive numbers}.
	Then
	$$
		SH^{S^1,+}(W,\Q) = \bigoplus_{\gamma\in\Per(R_{\alpha})} \Q\langle\gamma\rangle
	$$
	where $\Per(R_{\alpha})$ denotes the set of good periodic Reeb orbits on $\partial W$.
\end{theorem}

In Section \ref{chapter:properties}, we show that positive $S^1$-equivariant symplectic homology has good functorial properties.
In the first part, we construct a ``transfer morphism'' between all the above mentioned variants of symplectic homology when one  Liouville domain is embedded  into an other Liouville domain.
This construction generalises a construction given by Viterbo (\cite{V}).
We prove in theorem \ref{thm:compos} that this morphism has nice composition properties.
The second part of Section \ref{chapter:properties} is dedicated to the invariance of the  variants of symplectic homology.
This is not new, but a complete and self-contained proof is difficult to find in the literature.
We prove :
\begin{theorem}\label{thm:invLiouville}
	Let $(W_0,\lambda_0)$ and $(W_1,\lambda_1)$ be two Liouville manifolds\footnote{We refer to Definition \ref{def:liouvillemanifold} for a precise definition of Liouville manifold.}
	of finite type such that there exists a symplectomorphism $f:(W_0,\lambda_0)\rightarrow(W_1,\lambda_1)$.
	Then
	\[
		SH^\dagger(W_0,\lambda_0)\cong SH^\dagger(W_1,\lambda_1).
	\]
	where $\dagger$ can be any of the following symbol: $\emptyset, +, S^1, (S^1,+)$.
\end{theorem}
We prove that the positive $S^1$-equivariant symplectic homology yields an invariant of some contact manifolds in the following sense.
\begin{theorem}\label{prop:invtcontact}
	Let $(M_0,\xi_0)$ and $(M_1,\xi_1)$ be two contact manifolds that are exactly fillable; i.e. there exist Liouville domains $(W_0,\lambda_0)$ and $(W_1,\lambda_1)$ such that
	$\partial W_0 = M_0$, $\xi_0=\ker({\lambda_0}_{|_{M_0}})$, $\partial W_1=M_1$ and $\xi_1=\ker({\lambda_1}_{|_{M_1}})$.
	Assume there exists a contactomorphism $\varphi:(M_0,\xi_0)\rightarrow (M_1,\xi_1)$.
	Assume moreover that $\xi_0$ admits a contact form $\alpha_0$ such that all periodic Reeb orbits are nondegenerate and the set of Conley-Zehnder indices of all good periodic Reeb orbits is lacunary.
	Then
	\[
		SH^{S^1,+}(W_0,\lambda_0)\cong  SH^{S^1,+}(W_1,\lambda_1).
	\]
\end{theorem}
{Note that the isomorphism in Theorem \ref{prop:invtcontact} is between the graded modules.
There can be a shift in the degree of the generators (by even numbers).}
This Theorem, together with Theorem \ref{thm:computing}, allow us to reformulate in section \ref{section:Usti} the proof of Ustilovsky's result on the existence of non diffeomorphic contact structures on the spheres $S^{4m+1}$.
The original proof depends on a theory of cylindrical contact homology, which is not yet rigorously established due to transversality problems.
Another proof of this result using Rabinowitz-Floer homology was done independently by Fauck \cite{fauck}.
\begin{theorem}[\cite{U}]
	For each natural number $m$, there exist infinitely many pairwise non isomorphic contact structures on $S^{4m+1}$.
\end{theorem}

In Section \ref{chap:orbits} we use positive $S^1$-equivariant symplectic homology to get results about the minimal number of distinct
periodic Reeb orbits on some contact manifolds. We first give in Theorem \ref{ekeland} a new proof (in the non degenerate case) of 
 a Theorem by Ekeland and Lasry stating that  the minimal number of geometrically distinct periodic Reeb orbits on a contact type hypersurfaces $\Sigma$ in $\R^{2n}$
 is at least $n$  when the hypersurface is nicely pinched between two spheres of radius $R_1$ and $R_2$. We then use  
the techniques developed for this proof  to  study some hypersurfaces in some negative line bundles.
We first get a description  for circle bundles.
\begin{proposition}\label{propintro:orb}
	Let $\Sigma$ be a contact type hypersurface in a negative line bundle over a closed symplectic manifold $\mathcal{L}\rightarrow B$ such that the intersection of $\Sigma$ with each fiber is a circle.
	The contact form is the restriction of $r^2\theta^{\nabla}$ where $\theta^\nabla$ is the transgression form on $\mathcal{L}$ and $r$ is the radial coordinate on the fiber.
	Then $\Sigma$ carries at least $\sum_{i=0}^{2n} \beta_{i}$ geometrically distinct periodic Reeb orbits, where the $\beta_{i}$
	are the Betti numbers of $B$.
\end{proposition}
We study then hypersurfaces pinched between circles bundles and get:
\begin{theorem}\label{thmintro:orb}
	Let $\Sigma$ be a contact type hypersurface in a negative line bundle $\mathcal{L}$, over a symplectic manifold $B$.
	Suppose that there exists a Liouville domain $W'$ (such that its first Chern class vanishes on all tori)  whose boundary coincides with the circle bundle of radius $R_1$ in $\mathcal{L}$, denoted $S_{R_1}$.
	Suppose there exists a Morse function $f:B\rightarrow\R$ such that the set of indices of all critical points of $f$ is lacunary.
	Let $\alpha$ be the contact form on $\Sigma$  induced by $r^2\theta^\nabla$ on  $\mathcal{L}$ ($\theta^\nabla$ is the transgression form on $\mathcal{L}$ and $r$ is the radial coordinate on the fiber).
	Assume that $\Sigma$ is ``pinched'' between two circle bundles $S_{R_1}$ and $S_{R_2}$ of radii $R_1$ and $R_2$ such that $0<R_1<R_2$ and $\frac{R_2}{R_1}<\sqrt{2}$.
	Assume that  the minimal period of any periodic Reeb orbit on $\Sigma$ is  bounded below by $R_1^2$.
	Then $\Sigma$ carries at least $\sum_{i=0}^{2n} \beta_{i}$ geometrically distinct periodic Reeb orbits, where the $\beta_{i}$ denote the Betti numbers of $B$.
\end{theorem}
Examples of line bundles where  this  theorem can be applied are given : the tautological bundle over a complex projective space and the tautological bundle over the Grassmannian of oriented $2$-planes in $\R^{2n}$.

\subsection*{Thanks}
The results presented here are part of my PhD thesis; I  thank my thesis advisers 
Frédéric Bourgeois and Alexandru Oancea for their  guidance and
 enlightening remarks.
My work also greatly benefited from discussions with Strom Borman, Samuel Lisi, Michael Hutchings, 
Jungsoo Kang, Yiming Long, Maksim Maydanskiy, Vinicius Ramos and Alexander Ritter.\\
I am grateful to the  Belgian FNRS-FRS for its support through a position
of Aspirant during my thesis.
I  deeply thank the referee for his careful reading and remarks which lead to corrections of this text.
This paper was written while being a Postdoctoral Fellow of the BAEF at UC Berkeley under the supervision of Michael Hutchings.
I also acknowledge travel support from the ERC via grants StG-259118-STEIN 
and StG-239781-ContactMath, from the Communaut\'e fran\c caise de Belgique via an ARC  
and from the Belgian federal government via the PAI ``Dygest''.


\section{Background on symplectic homology}\label{chap:background}


\begin{definition}[The Setup]\label{def:setup}
Let $(W,\omega)$ be a compact {\emph{symplectic manifold with contact type boundary}} $M:=\partial W$.
This means that there exists a Liouville vector field $X$ (i.e. a vector field $X$ such that $\mathcal{L}_{X}\omega = \omega$) defined on a neighbourhood of the boundary $M$, and  transverse to $M$. In the sequel, we 
assume that the Liouville vector field has been chosen and  we   denote by $(W,\omega,X)$  such
a manifold.
We denote by $\lambda$ the $1$-form defined in a neighbourhood of $M$ by 
$
\lambda : =\iota(X)\omega
$
and by $\alpha$  the  {\emph{contact $1$-form}} on $M$ which is the restriction of $\lambda$ to $M$:
\[
	\alpha := \bigl(\iota(X)\omega\bigr)_{\vert_{M}} .
\]
We denote by $\xi$ the {\emph{contact structure}} defined by $\alpha$ ($\xi:=\ker\alpha$), and by $R_{\alpha}$ 
the {\emph{Reeb vector field}} on $M$ ($
\iota(R_{\alpha})d\alpha =0, ~
		\alpha(R_{\alpha}) = 1$).
The {\emph{action spectrum}} of $(M,\alpha)$ is the set of all periods of the Reeb vector field :
$$
	\Spec(M,\alpha) := \{ T\in\R^{+} \,\vert\, \exists \gamma \textrm{ periodic orbit of } R_{\alpha} \textrm{ of period } T \}.
$$
The {\emph{symplectic completion}}  of $(W,\omega,\lambda)$ is the symplectic manifold defined by
\[
	\widehat{W} := W \bigcup_{G} (M\times\R^{+}):=\bigl(W \sqcup (M\times [-\delta,+\infty])\bigr)
	\raisebox{-0.5ex}{$/$}\raisebox{-1ex}{$\sim_G$}
\]
with the symplectic form
\[
	\widehat{\omega} :=
	\begin{cases}
		\omega &\textrm{on } W\\
		d(e^{\rho}\alpha) &\textrm{on } M\times [-\delta,+\infty]
	\end{cases}.
\]
The equivalence $\sim_G$,  between  a neighbourhood $U$ of $M$ in $W$ and $M\times[-\delta,0]$, is defined by the diffeomorphism
$
	G: M\times[-\delta,0]\rightarrow U:\, (p,\rho)\mapsto \varphi_{\rho}^{X}(p)
$
which is induced by  the flow of the  vector field $X$.\\
{\emph{A Liouville domain}} is a compact {{symplectic manifold with contact type boundary}} $(W,\omega,X)$
for which the vector field $X$ is globally defined. In that case the symplectic $2$-form is exact : $\omega=d\lambda$ where
$\lambda =\iota(X)\omega$; a Liouville domain is also called an {\emph{exact symplectic manifold with contact type boundary}} 
and is sometimes denoted $(W,\lambda)$.\\
Throughout this paper we assume 
that $W$ is symplectically atoroidal, i.e the symplectic form and the first Chern class vanish on all tori.
This assumption implies that the action of a loop and the Conley-Zehnder index\footnote{For a definition of the Conley-Zehnder index, which will be denoted $\mu_{CZ}$ in the following, we refer to \cite{JGutt, AD, CZ}} of a $1$-periodic orbit of a Hamiltonian are well-defined. 
\end{definition}


Symplectic homology was developed by Viterbo in \cite {V}, using works of Cieliebak, Floer, Hofer \cite{FH, CFH}; it is 
defined for a compact symplectic manifold $W$ with boundary of contact type,
 as  a direct limit  of Floer homologies of the symplectic completion of $W$, using 
some special Hamiltonians. 

\begin{definition}\label{def:Hstd}
	The class $\mathcal{H}_{\textrm{std}}$ of {\emph{admissible Hamiltonians}}  consists of smooth functions $H : S^{1}\times\widehat{W} \rightarrow \R$
	satisfying the following conditions:\\
		\indent 1- $H$ is negative and $C^{2}$-small on $S^{1}\times W$ ;\\
		\indent 2- there exists $\ \rho_{0} \geq 0$ such that $H(\theta, p, \rho) = \beta e^{\rho} + \beta' \textrm{ for } \rho\geq \rho_{0}$, with
			$0<\beta\notin\Spec(M,\alpha)$  and $\beta'\in\R$;\\
		\indent 3- $H(\theta, p, \rho)$ is $C^{2}$-close to  $h(e^{\rho})$   on $S^{1}\times M\times[0,\rho_{0}]$, for $h$ a convex increasing function.\\
\noindent We say furthermore that it is non degenerate if all $1$-periodic orbits of $X_{H}$ are nondegenerate (for a time-dependent Hamiltonian $H : S^{1}\times\widehat{W} \rightarrow \R$, the time dependent Hamiltonian vector field $X_{H}$ is defined by the relation 
$\widehat{\omega}(X^{\theta}_{H},.) = dH(\theta,\cdot)$ for each $\theta\in S^1$).\\
We  denote  by  $\Per(H)$ the set of $1$-periodic orbits of $X_H$ . \\ 
	The {\emph{ class $\J$ of admissible}} $J: S^1\rightarrow \End (T\widehat{W}) : \theta \mapsto J^\theta$ consists of smooth loops of compatible almost complex structures $J^\theta$ on $\widehat{W}$,
	such that, at infinity (i.e. for $\rho$ large enough) $J$ is autonomous (i.e. independent of $\theta$),  invariant under translations in the $\rho$ variable,  and satisfies
	\[
		J^\theta \xi= \xi \qquad  J^\theta (\partial_{\rho})=R_{\alpha}.
	\]
\end{definition}
\begin{remark}\label{orbitsofHstand}
Condition 1 implies that the only $1$-periodic orbits of $X_H$ in $W$ are constants; they correspond to critical points of $H$.\\
On $S^{1}\times M\times[0,+\infty[ $, for a  Hamiltonian of the form $H_1(\theta, p, \rho)=h_1(e^{\rho})$, one has $X_{H_1}^{\theta}(p,\rho)=-h'_1(e^{\rho})R_{\alpha}(p)$.
 Hence, for such a Hamitonian $H_1$,  with $h_1$ increasing, the image of a $1$-periodic orbit of $X_{H_1}$ is the image of a periodic orbit of the Reeb vector field $-R_{\alpha}$ of period $T:=h'_1(e^{\rho})$ located at level $M\times \{ \rho \}$. In particular, condition 2 implies that there is no $1$-periodic orbit of $X_H$ in $M\times [\rho_0,+\infty[ $ for a Hamiltonian $H$ in $\mathcal{H}_{\textrm{std}}$. Condition 3 ensures that for any non constant $1$-periodic orbit $\gamma_H$ of $X_H$  for a Hamiltonian $H$ in $\mathcal{H}_{\textrm{std}}$, there exists a closed orbit of the Reeb vector field $R_{\alpha}$ of period $T<\beta$ (with $\beta$ the slope of $H$ ``at $\infty$''), such that $\gamma_H$ is close to this  closed orbit of (minus) the Reeb vector field  located in $M\times \{ \rho \}$
 with  $T=h'(e^{\rho})$. \\  
We can consider a larger class of admissible Hamiltonians, removing conditions 1 and 3.	
\end{remark}
 \begin{definition}[Symplectic homology]	
The Floer complex $ SC(H,J) $ is the complex generated by $1$-periodic orbits of the Hamiltonian vector field $X_H$, graded by minus their Conley Zehnder index 
, with boundary $\partial$ (well defined for a  so called regular pair $(H,J)$) defined  by a count with signs of  Floer trajectories,
i.e maps $u:\R\times S^1\rightarrow W$ satisfying:
\begin{equation} \label{Floer}
	\frac{\partial u}{\partial s}(s,\theta) + J^{\theta}\bigl(u(s,\theta)\bigr)\biggl(\frac{\partial u}{\partial \theta}(s,\theta) - X_H^{\theta}\bigl(u(s,\theta)\bigr)\biggr)=0.
\end{equation}
The {\emph{symplectic homology}}  of $(W,\omega,X)$ is defined as the direct limit 
$$
	SH_{*}(W,\omega,X) := \lim_{\substack{\longrightarrow \\ H\in \mathcal{H}_{\textrm{std}}}} SH_{*}(H,J)
$$
where, for each $H$,  $J:S^1\rightarrow End\bigl((\Gamma(T\widehat{W})\bigr)$  is chosen so that $(H,J)$ is a regular pair.
 To define the  direct limit  one needs a partial order $\leq$ on $\Hstd$ and morphisms  $SH(H_1,J_1)\rightarrow SH(H_2,J_2)$ whenever $H_1\leq H_2$ are non degenerate.
The partial order on $\Hstd$ is given by $H_1\leq H_2$ if $H_1(\theta,x)\leq H_2(\theta,x)$ for all $(\theta,x)\in S^1\times\widehat{W}$
		(for more general Hamiltonians, it is enough to have $H_1(\theta,x)\leq H_2(\theta,x)$ for all $(\theta,x)$ outside a compact domain).\\		
The  morphism $SH(H_1,J_1)\longrightarrow SH(H_2,J_2)$ is the  {\emph{continuation morphism}}   induced by an smooth  increasing homotopy of regular pairs $(H_s,J_s), s\in\R$ between $(H_1,J_1)$ and $(H_2,J_2)$, with $(H_s,J_s)$  constant for $\abs{s}$ large, when it is regular.  By increasing, we mean $\partial_sH_s\geq0$ (again for more general $H$ it is enough to consider $\partial_sH_s\geq0$ outside a compact subset). The continuation morphism  is defined 	by a count with signs of  solutions $u:\R\times S^1\rightarrow \widehat{W}$ of the  equation
\begin{equation}\label{eq:floerparam}
	\frac{\partial u}{\partial s} + J_s^{\theta}\circ u\Bigl(\frac{\partial u}{\partial \theta} - X^\theta_{H_s}\circ u\Bigr)=0
\end{equation}
 with finite energy $ E(u) := \int_{-\infty}^{+\infty}\int_{S^1}\norm{\frac{\partial u}{\partial s}}^2 d\theta ds$ .
 \end{definition}

The idea of positive symplectic homology is to ``remove'' the data of constant $1$-periodic orbits from symplectic homology.
{\emph{We assume  that $(W,\omega,X)$ is a Liouville domain},
in order to identify the set of critical points of a Hamiltonian with its $1$-periodic orbits of small action.\\
Let $H:S^1\times\widehat{W}\rightarrow \R$ be a Hamiltonian in $\Hstd$. The Hamiltonian action functional $\A_{H}: C_{\textrm{contr}}^{\infty}(S^{1},\widehat{W})\rightarrow \R$ is defined as
\[
	\A_{H}(\gamma) := -\int_{D^{2}}\sigma^{\star}{\widehat{\omega}} - \int_{S^{1}}H\bigl(\theta,\gamma(\theta)\bigr)d\theta
\]
where $\sigma : D^{2}\rightarrow\widehat{W}$ is an extension of $\gamma$ to the disc $D^{2}$.
When the symplectic form is exact, $\w=d\lambda$, the action is
$\A_{H}(\gamma) := -\int_{S^{1}}\gamma^{\star}{\widehat{\lambda}} - \int_{S^{1}}H\bigl(\theta,\gamma(\theta)\bigr)d\theta.$\\
The $1$-periodic orbits of $H \in \Hstd$ fall into two classes  :
critical points in $W$, whose action is strictly less than some small positive constant $\epsilon$ (indeed, if $(\theta,x)$ is a critical point of $H$, the action of the constant orbit is equal to $-H(\theta,x)$) and  non-constant periodic orbits lying in $\widehat{W}\setminus W$ whose action is strictly greater than $\epsilon$ ( indeed, the action of such an orbit  is close, for a given $\rho$ in $[0,\rho_0]$ with $T=h'(e^\rho)$ in $\Spec(M,\alpha)$, to the action of the orbit of the vector field $-h'(e^\rho)R_{\alpha}$ located in $M\times \{ \rho \}$; this is given by $-\int_{S^1}e^\rho\alpha(-h'(e^\rho)R_{\alpha})d\theta -\int_{S^1}h(e^\rho)d\theta=e^\rho h'(e^\rho)-h(e^\rho)=e^\rho T-h(e^\rho)$; it is positive since $h$ is convex).\\
The $\epsilon$  above is  chosen (for instance) as half the minimal value of the periods of closed orbits of the Reeb vector field on $M=\partial W$. Functions $H$ are chosen so that the value of $\vert H\vert$ in $S^1\times W$ is less than $\epsilon$,
so that $h(e^\rho)$ is less than $\half \epsilon$ (hence $e^\rho T-h(e^\rho)$ is greater than $\frac{3}{2}\epsilon$) and the $C^2$-closeness to an autonomous function  is such that the  actions differ at most by $\half \epsilon$.
\begin{definition}[positive symplectic homology] Let $(W,\omega,X)$ be a Liouville domain and let $H$ be in $\Hstd$.
Let $SC^{\leq\epsilon}(H,J)$ be the complex generated by the $1$-periodic orbits of action no greater than $\epsilon$.
It is built out of critical points of $H$ and it is a subcomplex of $SC(H,J)$, since the action decreases along Floer trajectories.
	The {\emph{positive Floer complex}} is defined as the quotient of the total complex by the subcomplex of critical points;
	$$
		SC^{+}(H,J) := SC(H,J)/\raisebox{-1ex}{$SC^{\leq\epsilon}(H,J)$}
	$$
The differential induces a differential on the quotient which is still denoted $\partial$.
The continuation morphisms mentioned above descend to the quotient since the action decreases along a solution of \eqref{eq:floerparam}  (when the homotopy is increasing everywhere).
The {\emph{positive symplectic homology}} of $(W,\omega)$ is defined as a direct limit over non degenerate $H\in\Hstd$ of the homology of $SC^{+}(H,J)$
	$$
		SH^{+}(W,\omega,X) := \lim_{{\substack{\longrightarrow\\ {H\in \Hstd}}}} H_{*}\bigl(SC^{+}_{*}(H,J),\partial\bigr).
	$$
\end{definition}

\subsection{$S^1$-equivariant symplectic homology}\label{section:SHS1}

Let $X$ be a topological space endowed with an $S^1$-action.
If the $S^1$-action is free, $X/\raisebox{-0.5ex}{$S^1$}$ is a topological space.
The aim of $S^1$-equivariant homology is to build on the space $X$ a homology which coincides, when the action is free, with the singular homology of the quotient.
 One considers the universal principal $S^1$-bundle $ES^1\rightarrow BS^1$. 
The diagonal action on $X\times ES^1$ is free and one denotes by $X\times_{S^1} ES^1$ the quotient $(X\times ES^1)/\raisebox{-0.5ex}{$S^1$}$.
	Following Borel, the {\emph{$S^1$-equivariant homology of $X$}} with $\Z$-coefficients is defined as
	$
		H^{S^1}_*(X):=H_*(X\times_{S^1} ES^1,\Z).
	$
For symplectic homology, one defines in a similar way the $S^1$-equivariant symplectic homology
for any compact symplectic manifold  with contact type boundary $(W,\w,X)$;
the $S^1$-action one is referring to  is the reparametrization action on the loop space (not an action on $W$)
\[
	\varphi\cdot\gamma(\theta) = \gamma(\theta-\varphi).
\]
This homology was first introduced by Viterbo in \cite{V}; we present here a different  approach, which 
 was sketched by Seidel in \cite{Seidel} and which was  studied in detail by Bourgeois and Oancea  in \cite{bo, BOjems, BOind}.
 It has the advantage to use a special class of Hamiltonians and simplified equations for Floer trajectories, so that computations are often doable. 
 The important point is that this $S^1$equivariant symplectic homology coincides with the one defined by Viterbo.

\medskip
The model of  $ES^1$ is given as a limit of spheres  $S^{2N+1}$ for $N$ going to $\infty$ with the Hopf $S^1$-action.
To provide $S^1$-invariant functionals, we use $S^1$-invariant Hamiltonians : $H:S^1\times\widehat{W}\times S^{2N+1}\rightarrow\R$ such that
$$ 
H(\theta+\varphi,x,\varphi z) = H(\theta,x,z)\qquad \forall \theta,\varphi\in S^1, \forall z\in S^{2N+1}.
$$
The {\emph{parametrised action functional}} $\A:C_{contr}^{\infty}(S^1,\widehat{W})\times S^{2N+1}\longrightarrow\R$,  is defined as
\begin{equation}\label{eq:paramaction}
	\A(\gamma,z):=-\int_{D^{2}}\sigma^{\star}\widehat{\omega}-\int_{S^1}H\bigl(\theta,\gamma(\theta),z\bigr)d\theta 
\end{equation}
where $\sigma : D^{2}\rightarrow\widehat{W}$ is an extension of $\gamma$ to the disc $D^{2}$.
It is invariant under the diagonal $S^1$-action on $C^{\infty}(S^1,\widehat{W})\times S^{2N+1}$.\\
The {\emph{critical points}} of the parametrised action functional are pairs $(\gamma,z)$ such that
\begin{equation}\label{critpointsS1}
	\gamma\in\Per(H_z)\quad \textrm{and}\quad\int_{S^1}\frac{\partial H_z}{\partial z}\bigl(\theta,\gamma(\theta)\bigr)d\theta=0,
\end{equation}
where $H_z$ is the function on $S^1\times\widehat{W}$ defined by $H_z(\theta,x):=H(\theta,x,z)$ and where
$\Per(H_z)$ denote, as before, the set of $1$-periodic orbits of $X_{H_z}$.
The set $\Per^{S^1}(H)$ of critical points of $\A$  is $S^1$-invariant.
If $q=(\gamma,z)\in\Per^{S^1}(H)$, we denote by $S_q$ the $S^1$-orbit of $q$.
Such an $S_q$ is called {\emph{nondegenerate}} if the Hessian $d^2\A(\gamma,z)$ has a $1$-dimensional kernel for some (and hence any) $(\gamma,z)\in S_q$.

\medskip
\noindent{\bf{The data : }}
We consider a compact symplectic manifold with compact type boundary $(W,\w,X)$.
We  fix a sequence of $C^2$-small perfect Morse functions $f_N:\C P^N\rightarrow\R$,  together with a  Riemannian metric $\bar{g}_N$ on $\C P^N$ for which the gradient flow of $f_N$ has the Morse-Smale property.
We shall take  the standard metric and
\[
	f_N\bigl([w^0:\ldots:w^N]\bigr) = C\frac{\sum_{j=0}^N (j+1)\abs{w^j}^2}{\sum_{j=0}^N\abs{w^j}^2} \quad\textrm{  with } C<0\in\R.
\]
We denote by $\tilde{f}_N:S^{2N+1}\rightarrow \R$ their $S^1$-invariant lift,  and by $Crit(\tilde{f}_N)$  the set of critical points of $\tilde{f}_N$ (which is a union of circles).
We choose a point $z_j$ on the critical circle which projects on the critical point of $-f_N$ of index $2j$.
In our example,  $z_j=(w^0,\ldots,w^N)\in S^{2N+1}$  with $w^i=\delta^i_j$.
We fix a local slice $T_{z_j}$ transverse in $S^{2N+1}$ to the circle in $Crit(\tilde{f}_N)$ at  $z_j$; again in our example $T_{z_j} = \{(w^0\ldots,w^N)\in S^{2N+1}\,|\,w^j\in\R^+\}$.
We consider $U_N$ a neighbourhood of $Crit(\tilde{f}_N)$ and  $\check{\rho}_N : S^{2N+1}\rightarrow\R$  a $S^1$ invariant cut-off function on $U_N$ which is equal to $1$ in a neighbourhood
$U_N'\subset U_N$ of $Crit(\tilde{f_N})$ and $0$ outside $U_N$.
We set  $$\epsilon_N := \min_{z\in S^{2N+1}\setminus U_N'} \norm{\vec{\nabla}\tilde{f_N}(z)} >0.$$
\begin{definition}[Class of admissible Hamiltonians]\label{def:hamadmS1}
An {\emph{$S^1$-invariant Hamiltonian $H$ is  admissible}} if $H_z$ is in $\Hstd$ (as in Definition \ref{def:Hstd}) with constant slope independent of $z$ for all $z\in S^{2N+1}$
	and if for any critical point $q\in\Per^{S^1}(H)$, the $S^1$-orbit $S_q$ is non degenerate.
	Let $\mathcal{H}^{S^1,N}$ be the family of such Hamiltonians.	
We look at the subfamily $\Hs^{S^1,N}(f_N)\subset\Hs^{S^1,N}$ consisting of  Hamiltonians of the form $H+\tilde{f}_N$ with $H : S^1\times \widehat{W}\times S^{2N+1}\rightarrow \R$ in $\Hs^{S^1,N}$ such that\\
	\indent 1. Each critical point $(\gamma,z)$ of the parametrized action functional $\A_{H+\tilde{f}_N}$ defined by ${H+\tilde{f}}_N$ lies over a $z$ which is a critical point of $\tilde{f}_N$;\\
	\indent 2. For every $z\in Crit(\tilde{f}_N)$, $H(\cdot,\cdot,z)$ has non degenerate periodic orbits;\\
	\indent 3. \label{condnondeg} $H+\tilde{f}_N$ has nondegenerate $S^1$-orbits;\\
 	\indent 4. \label{condprocheMB} $\norm{\vec{\nabla}_zH\bigl(\theta,x,z\bigr)}<\epsilon$, for all $z\in S^{2N+1}\setminus U'$;\\
	\indent 5. \label{condspe} For all $z\in U'$, $\vec{\nabla}_zH\cdot\vec{\nabla}\tilde{f}_N(z)=0$.
\end{definition}
\begin{remark}
	Condition 3 can be replaced by the following :
	near every critical orbit of $\tilde{f}_N$, we have $H(\theta,x,z) = H'(\theta-\phi_z,x)$, where $\phi_z\in S^1$ is the unique element such that
	the action of its inverse brings $z$ into $T_{z_0}$, i.e. $\phi_z^{-1}\cdot z\in T_{z_0}$ and $H'\in\Hstd$.
	We shall consider  elements $H$ which are built from an $H' : S^1\times \widehat{W}\rightarrow \R$ in $\Hstd$ as in Definition \ref{def:Hstd}, close to an autonomous Hamiltonian; we shall develop this in next section.
\end{remark}

\begin{definition}[The chains]
Given an admissible $H+\tilde{f}_N$, the set denoted $\Per^{S^1}(H+\tilde{f}_N)$ of critical points $(\gamma, z)$ of the parametrized action functional $\A_{H+\tilde{f}_N}$ is a union of  circles 
$$
	\Big\{ S_{(\gamma,z)}:=S^1\cdot\bigl(\gamma,z\bigr)= \{\varphi\cdot(\gamma,z)=(\varphi\cdot\gamma,\varphi z)\,|\,\varphi\in S^1\}\Big\}.
$$
Each  of those circle gives a generator of the chain complex.
The index of the generator $S_{(\gamma,z)}$ is defined to be 
$$\mu(S_{\gamma,z}):=-\mu_{CZ}(\gamma)+\mu_{Morse}(z;-\tilde{f}_N).$$
The {\emph{chain complex}} is defined as:
$$
	S\widetilde{C}_*^{S^1,N}(H,f_N) := \bigoplus_{S_p\subset\Per^{S^1}(H+\tilde{f}_N)}\Z\langle S_p\rangle.
$$
\end{definition}
\begin{definition}[The differential] 
A {\emph{parametrized loop }} of almost complex structures
$
	J:S^1\times S^{2N+1}\rightarrow\End(T\widehat{W}), (\theta,z)\mapsto J^{\theta}_z
$
 is   $S^1$-invariant if $
	J^{\theta+\varphi}_{\varphi z} = J^{\theta}_z,\, \forall \theta,\varphi\in S^1, \forall z\in S^{2N+1} $
and is {\emph{admissible}} if for all $z$ in $S^{2N+1}$, the loop of almost complex structures $J_z$ is in $\J$ as defined in Definition \ref{def:Hstd}.\\
Let $(J^{\theta}_z)$ be an $S^1$-invariant family of almost complex structures independent of $z$ along each local slice.
Let $p^-=(\gamma^-,z^-)$ and $p^+=(\gamma^+,z^+)$ be two critical points of $\A_{H+\tilde{f}_N}$.
We denote by $\widehat{\Mod}(S_{p^-},S_{p^+};H,f_N,J^{\theta}_z,g)$ the space of solutions $(u,z)$, $u:\R\times S^1\rightarrow\widehat{W}$, $z:\R\rightarrow S^{2N+1}$
to the system of equations
\begin{equation}
	\left\{
	\begin{aligned}
		\partial_su+J^{\theta}_{z(s)}\circ u\bigl(\partial_{\theta}u-X_{H_{z(s)}}\circ u\bigr) &=0\\
		\dot{z}-\vec{\nabla}\tilde{f}_N(z) &=0
	\end{aligned}
	\right.
\end{equation} such that $
	\lim_{s\rightarrow-\infty}\bigl(u(s,\cdot),z(s)\bigr) \in S_{p^-}$ and $ \lim_{s\rightarrow\infty}\bigl(u(s,\cdot),z(s)\bigr) \in S_{p^+}$.\\
If $S_{p^-}\neq S_{p^+}$, we denote by ${\Mod}(S_{p^-},S_{p^+};H,f_N,J^{\theta}_z,g)$ the quotient of the space\\
 $\widehat{\Mod}(S_{p^-},S_{p^+};H,f_N,J^{\theta}_z,g)$ by the reparametrization $\R$-action.
This quotient  ${\Mod}(S_{p^-},S_{p^+};H,f_N,J^{\theta}_z,g)$ carries a free $S^1$-action and we denote by $\Mod^{S^1}(S_{p^-},S_{p^+};H,f_N,J^{\theta}_z,g)$ the quotient of\\
${\Mod}(S_{p^-},S_{p^+};H,f_N,J^{\theta}_z,g)$ by this $S^1$-action.
For generically chosen $J^{\theta}_z$ and $g$, it is proven in \cite{bo} that the spaces 
	$\Mod^{S^1}(S_{p^-},S_{p^+};H,f_N,J^{\theta}_z,g)$ are smooth manifolds of dimension equal to $-\mu(S_{p^-})+\mu(S_{p^+})-1$.\\
The {\emph{differential}} $\widetilde{\partial}^{S^1} : S\widetilde{C}_*^{S^1,N}(H,f_N)\rightarrow S\widetilde{C}_{*-1}^{S^1,N}(H,f_N)$ is defined by
$$
	\widetilde{\partial}^{S^1}(S_{p^-}) := \sum_{\substack{S_p^+\subset\Per^{S^1}(H+\tilde{f}_N)\\ \mu(S_{p^-})-\mu(S_{p^+})=1}} \#\Mod^{S^1}(S_{p^-},S_{p^+};H,f_N,J^{\theta}_z,g)S_{p^+}
$$
where $ \#$ is a count with signs defined in \cite{bo}.
\emph{Continuation maps} are defined as usual, using the space of solutions $(u,s)$ of
\begin{equation}\label{eq:continuationS12}
	\left\{
	\begin{aligned}
		\partial_su+J^{\theta}_{s,z(s)}\circ u\bigl(\partial_{\theta}u-X_{H_{s,z(s)}}\circ u\bigr) &=0\\
		\dot{z}-\vec{\nabla}\tilde{f}_N(z) &=0
	\end{aligned}
	\right.
\end{equation}
with $H_s+\tilde{f}_N$ an increasing homotopy between $H_0+\tilde{f}_N$ and $H_1+\tilde{f}_N$.
\end{definition}
\begin{definition}[$S^1$-equivariant symplectic homology] ~\\
The {\emph{$S^1$-equivariant Floer homology groups}} are 
	\[
		SH^{S^1,N}_*(H,f_N, J,g) := H_*\bigl(S\widetilde{C}^{S^1,N}_*(H,f_N),\widetilde{\partial}^{S^1}\bigr).
	\]
	The {\emph{$S^1$-equivariant symplectic homology groups of $W$}} are 
	\[
		SH^{S^1}_*(W,\w,X) := \lim_{\substack{\longrightarrow\\ N}}\lim_{\substack{\longrightarrow\\ H\in\Hs^{S^1,N}(f_N)}}SH^{S^1,N}_*(H,f_N,J,g).
	\]
 The direct limit over $N$ is taken with respect to the $S^1$-equivariant embeddings $S^{2N+1}\hookrightarrow S^{2N+3}$ which induce maps
$$SH^{S^1,N}_*(W,\w,X)\rightarrow SH^{S^1,N+1}_*(W,\w,X)\quad \textrm{for each}\quad  N.$$
\end{definition}
We show here below that the action decreases along these continuation maps.
This allows to define $SH^{S^1,+}$ in the context of Liouville domains.
\begin{proposition}\label{prop:actionS1bo}
	Let $H_0+\tilde{f}_N$ and $H_1+\tilde{f}_N$ be Hamiltonians in $\Hs^{S^1,N}(f)$ and let $\tilde{H}_s:=H_s+\tilde{f}_N$ be an increasing homotopy between $H_0+\tilde{f}_N$ and $H_1+\tilde{f}_N$.
	If $(u,z)$, $u:\R\times S^1\rightarrow \widehat{W}$ and $z:\R\rightarrow S^{2N+1}$ is a solution of equations (\ref{eq:continuationS12})
	satisfying  the conditions ${\lim}_{{s\rightarrow-\infty}}\bigl(u(s,\cdot),z(s)\bigr)=\bigl(\gamma^-(\cdot),z^-\bigr)$ and $ \lim_{s\rightarrow+\infty}\bigl(u(s,\cdot),z(s)\bigr)=\bigl(\gamma^+(\cdot),z^+\bigr)
	$,
	then
	\[
		\A(\gamma^-,z^-)\geq\A(\gamma^+,z^+).
	\]
\end{proposition}
\begin{proof}
The parametrized action for the Hamiltonian $H_s+\tilde{f}$ on the pair $(u(s,\cdot ), z(s))$ is given by
\[ 
 -\int_{D^{2}}\sigma_s^{\star} \widehat{\w}
-\int_{S^1} ({H}_s+\tilde{f}_N) \left(\theta,u(s,\theta),z(s)\right)d\theta
 \]
where $\sigma_s : D^{2}\rightarrow\widehat{W}$ is an extension of $\gamma_s=u(s,\cdot)$ to the disc $D^{2}$.
By the asphericity condition, $\int_{D^{2}}\sigma_s^{\star} \widehat{\w}=\int_{D^{2}}\sigma_{s_0}^{\star} \widehat{\w}
+\int_{S^1\times [s_0,s]}u^{\star} \widehat{\w}$ so that 
{\small{\begin{align*}
		&\frac{\partial}{\partial s}\A_{{H}_s+\tilde{f}_N}\bigl(u(s,\cdot),z(s)\bigr) =  -\int_{S^1}\w(\partial_su,\partial_{\theta}u)d\theta-\int_{S^1}\tfrac{\partial}{\partial u}{H}_s\bigl(\theta,u(s,\theta),z(s)\bigr)
			\tfrac{\partial}{\partial s}u(s,\theta)d\theta\\
			&\qquad~ -\int_{S^1}\vec{\nabla}_z({H}_s+\tilde{f}_N)\bigl(\theta,u(s,\theta),z(s)\bigr)\cdot\dot{z}(s)d\theta
			-\int_{S^1}\bigl(\tfrac{\partial}{\partial s}({H}_s+\tilde{f}_N)\bigr)\bigl(\theta,u(s,\theta),z(s)\bigr)d\theta\\
		&\qquad=  -\int_{S^1}\w(\partial_su,\partial_{\theta}u)d\theta- \int_{S^1}dH_{s,z(s)}(\partial_su)d\theta\\
			 &\qquad~ -\int_{S^1}\vec{\nabla}_z({H}_s+\tilde{f}_N)\bigl(s,\theta,u(s,\theta),z(s)\bigr)\cdot\vec{\nabla}f_N(z) d\theta
			-\int_{S^1}\tfrac{\partial}{\partial s}({H}_s+\tilde{f}_N)\bigl(\theta,u(s,\theta),z(s)\bigr)d\theta.
			\end{align*}}}
The last term  is $ \leq0$ since the homotopy is increasing. The first line can be rewritten as 
{\small{\begin{align*}
		-\int_{S^1}\w(\partial_su,\partial_{\theta}u)d\theta  - \int_{S^1}\w(X_{H_{s,z(s)}},\partial_su)d\theta
			&= -\int_{S^1}\w(\partial_su,\partial_{\theta}u-X_{H_{s,z(s)}})d\theta \\
		&= -\int_{S^1}\w(\partial_su, J^{\theta}_{z(s)}\partial_su)d\theta \\
		&=-\norm{\partial_su}^2_{g_{J^{\theta}_{z(s)}}} 	\leq0.
	\end{align*}}}
The first term in the second line is $ \leq0$ by conditions 4 and 5 in definition \ref{condprocheMB}  and by the definition of $\epsilon$.
\end{proof}
\begin{remark}\label{rmk:action}
	With the assumptions of Proposition, \ref{prop:actionS1bo}, it appears in the proof above that
$\int \norm{\partial_su}^2_{g_{J^{\theta}_{z(s)}}} dsd\theta \leq \A(\gamma^-,z^-)-\A(\gamma^+,z^+).$
\end{remark}

\begin{definition}[Positive $S^1$-equivariant symplectic homology]
	Let $H\in\Hs^{S^1,N}(f_N)$ be a Hamiltonian.
	The {\emph{positive $S^1$-equivariant complex}} is defined as
	\[
		S{\widetilde{C}}^{S^1,N,+}(H,f_N) := S{\widetilde{C}}^{S^1,N}(H,f_N) / \raisebox{-1ex}{$S{\widetilde{C}}^{S^1,N,\leq\epsilon}(H,f_N)$}
	\]
	where $S{\widetilde{C}}^{S^1,N,\leq\epsilon}(H,f_N)$ is the set of critical points of $\A_{H+\tilde{f}_N}$ of action less than $\epsilon$.
	The differential passes to the quotient where we still denote it ${\widetilde{\partial}}^{S^1}$ and the positive $S^1$-equivariant Floer groups are defined as
	\[
		SH^{S^1,N,+}(H,f_N) := H\bigl(S{\widetilde{C}}^{S^1,N,+}(H,f_N),{\widetilde{\partial}}^{S^1}\bigr).
	\]
	The {\emph{positive $S^1$-equivariant symplectic homology}} is defined by
	\[
		SH^{S^1,+}_*(W,\w,X) := \lim_{\substack{\longrightarrow\\ N}}\lim_{\substack{\longrightarrow\\ H\in\Hs^{S^1,N}(f_N)}}SH^{S^1,N,+}_*(H, f_N).
	\]
	We assume $(W,\w,X)$ to be exact and we assume the function $f_N$ to be small  in order to  identify $1$-periodic orbits of small action with a pair $(p,z)$, $p$ a critical points of $H$.
\end{definition}


\section{$SH^{S^1,+}$ and periodic Reeb orbits}\label{chap:SHS1orbits}

The goal of this section is prove theorem \ref{thm:computing} which  relates the positive $S^1$-equivariant homology of a Liouville domain $(W,\lambda)$
to the Reeb orbits on $(M=\partial W,\alpha=\lambda_{\vert_M})$.
\subsection{The multicomplex defining positive $S^1$-equivariant homology}
We use the nice subclass of Hamiltonians  introduced in \cite{bo}; they are
constructed using elements in $\Hstd$ which are small perturbations of autonomous Hamiltonians.
\begin{definition}[Construction of admissible Hamiltonians from elements in $\Hstd$]\label{constrS1inv}
For $H'\in\Hstd$, we define an $S^1$-invariant Hamiltonian  $H_N: S^1\times \widehat{W}\times S^{2N+1}\rightarrow \R$ as follows.
Define $\widetilde{H}_N : S^1\times\widehat{W}\times U_N\rightarrow \R$ by $\widetilde{H}_N(\theta,x,z) := H'(\theta-\phi_z,x)$ where $\phi_z\in S^1$ is the unique element such that $\phi_z^{-1}\cdot z\in T_{z_j}$
when $z$ is close to the critical circle including $z_j$,
and  extend $\widetilde{H}_N$ to $H_N:S^1\times\widehat{W}\times S^{2N+1}\rightarrow \R$, by
\begin{equation}\label{eq:constr}
	H_N(\theta,x,z) := \check{\rho}_N(z)\widetilde{H}_N(\theta,x,z) + \bigl(1-\check{\rho}_N(z)\bigr)\beta(x)H'(\theta,x)
\end{equation}
using the cutoff function $\check{\rho}_N$ on $S^{2N+1}$ and  a  function
 $\beta:\widehat{W}\rightarrow\R$  which is $0$ where $H'$ is time-dependent and equal to $1$ outside a compact set.
The element $H_N$ is  automatically in $\Hs^{S^1,N}(f_N)$,  when  $H'\in\Hstd$ is a   small perturbation of some autonomous functions as developed further in this section.
\end{definition}

\subsubsection*{The complex for a subclass of special Hamiltonians}\label{section:specialcomplex} ~\\
Let $H': S^1\times \widehat{W}\rightarrow \R$ in $ \Hstd$  be fixed, with non degenerate $1$-periodic orbits, and  consider a sequence $H_N\in\Hs^{S^1,N}, N\geq1$ such that $$H_N(\theta,x,z)=H'(\theta-\phi_z,x)\quad \textrm{for every }\quad z\in Crit(\tilde{f}_N)$$ (for instance by the construction above)
and a sequence $J_N\in\J^{S^1,N}$ such that $J_N$ is regular for $H_N$.

Let $i_0:\C P^N\hookrightarrow\C P^{N+1}: [w^0:\ldots:w^{N-1}]\mapsto[w^0:\ldots:w^{N-1}:0]$ and let $i_1:\C P^N\hookrightarrow\C P^{N+1}: [w^0:\ldots:w^{N-1}]\mapsto[0:w^0:\ldots:w^{N-1}]$
and denote by $\tilde{i}_0:S^{2N+1}\rightarrow S^{2N+3}:z\mapsto(z,0)$ and $\tilde{i}_1:S^{2N+1}\rightarrow S^{2N+3}:z\mapsto(0,z)$ their lifts.
Observe that
 $\im(i_0)$ and $\im(i_1)$ are invariant under the gradient flow of $f_{N+1}$,
 $f_N=f_{N+1}\circ i_0=f_{N+1}\circ i_1+cst$ and  $i_1^{\star}\bar{g}_{N+1}=i_0^{\star}\bar{g}_{N+1}=\bar{g}_N$. We assume furthermore that
	 $H_{N+1}\bigl(\cdot,\cdot,\tilde{i}_1(z)\bigr) = H_{N+1}\bigl(\cdot,\cdot,\tilde{i}_0(z)\bigr)=H_N(\cdot,\cdot,z)$, and also
 that  $J_{N+1,\tilde{i}_1(z)}=J_{N+1,\tilde{i}_0(z)}=J_{N,z}$.
The critical points of $\A_{H_N+\tilde{f}_N}$ are pairs $(\gamma_z,z)$ where $z$ is a critical point of $\tilde{f}_N$ and where $\gamma_z$ is a $\phi_z$-translation of a $1$-periodic orbit $\gamma$ of $H'$ in $\widehat{W}$
(i.e $\gamma_z(\theta)=\gamma(\theta-\phi_z)$ which writes $\gamma_z=\phi_z\cdot\gamma$).
There is  thus a natural identification (with gradings)
\begin{align*}
	S{\widetilde{C}}_*^{S^1,N}(H_N,f_N) &\simeq \Z[u]/\raisebox{-1ex}{$u^{N+1}$}\otimes_{\Z}SC_*(H',J)\\
	S^1\cdot(\gamma_{z_j},z_j) &\mapsto u^j\otimes\gamma =: u^j\gamma 
\end{align*}
where $z_j$ is the chosen critical point of $-\tilde{f}_N$ of index $2j$ and $u$ is a formal variable of degree $2$.\\
The differential, under this identification of complexes, writes 
\begin{equation}\label{eq:diffS1}
	{\widetilde{\partial}}^{S^1}(u^l\otimes\gamma)=\sum_{j=0}^l u^{l-j}\otimes  \varphi_j(\gamma).
\end{equation}
for maps $
	\varphi_j:S{{C}}_*(H')\rightarrow S{{C}}_{*+2j-1}(H')$
defined by counting with signs the elements of the space $\Mod^{S^1}(S_{(\gamma^-_{z_j},z_j)},S_{(\gamma^+_{z_0},z_0)};H_N,f_N,J_N,g_N)$ which is the quotient by the $\R$ and the $S^1$-action of the space of solutions of 
\[	\left\{
	\begin{aligned}
		\partial_su+J^{\theta}_{z(s)}\circ u\bigl(\partial_{\theta}u-X_{H_{N,z(s)}}\circ u\bigr) &=0\\
		\dot{z}-\vec{\nabla}\tilde{f}_N(z) &=0
	\end{aligned}
	\right.
\]
 going from $S^1 \cdot(\gamma^-_{z_j},z_j)$ to $S^1 \cdot(\gamma^+_{z_0},z_0)$.\\
It follows from the assumptions  that for a fixed $j$, the maps $\varphi_j$ obtained for varying values of $N\geq j$ coincide.
Therefore  the limit as $N\rightarrow\infty$ of all the $S{\widetilde{C}}_*^{S^1,N}(H_N,f_N)$ is encoded into a complex denoted
\[
	S{\widehat{C}}_*^{S^1}(H') := \Z[u]\otimes_{\Z}SC_*(H')
\]
with differential induced  by \eqref{eq:diffS1} that we can formally write as
\[
	{\widehat{\partial}}^{S^1} = \varphi_0+u^{-1}\varphi_1+u^{-2}\varphi_2+\ldots
\]
As before, there are well-defined continuation maps induced by increasing homotopies of Hamiltonians.
\begin{proposition}\cite{bo}\label{shortS1} The $S^1$ equivariant homology of $W$ is given by:
\[
	SH_*^{S^1}(W) := \lim_{\substack{\longrightarrow \\ H\in\Hstd}}H\Bigl( S{\widehat{C}}_*^{S^1}(H'),{\widehat {\partial}}^{S^1} \Bigr).
\]
\end{proposition}
\begin{definition}[Perturbation of Morse-Bott Hamiltonians]\label{sec:perturbationMB} \cite{BOduke}
Let $(W,\lambda)$ be  a Liouville domain.
The set of Morse-Bott Hamiltonians $\mathcal{H}_{\textrm{MB}}$ \label{def:HMB} consists of Hamiltonians $H: \widehat{W}\rightarrow \R$ so that\\
	\indent 1. \label{Hmorse} $H_{\vert_{W}}$ is a negative $C^{2}$-small Morse function;\\
\indent  2. $H(p,\rho) = h(\rho)$ outside $W$, where $h$ is a strictly increasing function, and $h(\rho) = \beta e^{\rho} + \beta'$ for $\rho>\rho_{0}$, where $\beta,\beta'\in\R$ and $\beta \notin \Spec(M,\alpha)$,
			 and we assume that $h''-h'>0$ on $[0,\rho_{0})$.\\
The $1$-periodic orbit of $X_{H}$ are either critical points of $H$ in $W$ or non constant $1$-periodic orbits,  located on levels
$M\times\{\rho\}, \rho\in (0,\rho_{0})$, which are in correspondence with periodic $-R_{\alpha}$-orbits of period $e^\rho h'(\rho)$.
Since $H$ is autonomous, every $1$-periodic orbit, $\gamma_H$ of $X_{H}$, corresponding to the periodic Reeb orbit $\gamma$,  gives birth to a $S^{1}$ family of $1$-periodic orbits of $X_{H}$, denoted  by $S_{\gamma}$.\\
An element $H\in\Ham$ is deformed, as in \cite{CFHW},  into a time-dependent Hamiltonian $H_{\delta}$ with only non degenerate $1$-periodic orbits
in the following way.
We choose  a perfect Morse function on the circle, $\check{f}: S^1\rightarrow \R$.
For each  $1$-periodic orbit  $\gamma_{H}$ of $X_{H}$, we consider the integer $l_{\gamma_H}$ so that $\gamma_H$ is a $l_{\gamma_H}$-fold cover of a simple periodic orbit;
$l_{\gamma_H}$ is constant on $S_\gamma$ and
we set $l_{\gamma}=l_{\gamma_H}=\frac{1}{T}$ where $T$ is the period of $\gamma$.
We choose a symplectic trivialization $\psi := (\psi_{1},\psi_{2}): U_{\gamma} \rightarrow V\subset S^{1}\times\R^{2n-1}$
from $U_{\gamma}\subset\partial W\times \R^+\subset \widehat{W}$,open neighborhood  of the image of $\gamma_{H}$, to $V$,open neighborhood   of $S^{1}\times\{0\}$,
such that $\psi_{1}\bigl(\gamma_H(\theta)\bigr) = l_{\gamma}\theta$.
Here $S^{1}\times\R^{2n-1}$ is endowed with the standard symplectic form.
Let $\check{g}:S^{1}\times\R^{2n-1}\rightarrow[0,1]$ be a smooth cutoff function supported in a small neighborhood of $S^{1}\times\{0\}$
such that $\check{g}_{\vert_{S^{1}\times\{0\}}} \equiv 1$.
We denote by $\check{f}_\gamma$ the function defined on $S_{\gamma}$ by $\check{f}\circ \psi_{1\vert_{S_{\gamma}}}$.\\
For $\delta >0$ and $(\theta,p,\rho)\in S^{1}\times U_{\gamma}$, we define
\begin{equation}\label{eq:perturbationMB}
	H_{\delta}(\theta,p,\rho) := h(\rho) + \delta \check{g}\bigl(\psi(p,\rho)\bigr)\check{f}\bigl(\psi_{1}(p,\rho)-l_{\gamma}\theta\bigr).
\end{equation}
The Hamiltonian $H_{\delta}$ coincides with $H$ outside the open sets $S^{1}\times U_{\gamma}$.
\end{definition}
\begin{lemma}[\cite{CFHW,BOduke}]\label{lem:orbitesMB}
The $1$-periodic obits of $H_{\delta}$, for   $\delta$ small enough, are either
constant orbits (the same as those of $H$) or nonconstant orbits which are non degenerate and form pairs ($\gamma_{\m},\gamma_{\M}$) which  coincide with the orbits in $S_{\gamma}$ starting at the
minimum and the maximum of $\check{f}_{\gamma}$ respectively, for each Reeb orbit $\gamma$ such that $S_{\gamma}$
appears in the $1$-periodic orbits of $H$.
Their Conley-Zehnder index is given by $\mu_{CZ}(\gamma_{\m}) = \mu_{CZ}(\gamma)-1$ and $\mu_{CZ}(\gamma_{\M}) = \mu_{CZ}(\gamma)$.
\end{lemma}

\subsection{Computing $SH^{S^1,+}$}\label{section:computS1}
We consider now the symplectic homologies with coefficients in $\Q$, denoted $SH^{\dagger}(W,\Q)$ on a Liouville domain $(W,\lambda)$.
 We consider a Hamiltonian denoted ${{H}}_{\delta,N}$ which is a  $S^1$-equivariant lift, as given by  formula \ref{eq:constr}, of a Hamiltonian  $ {H}_{\delta}$  which is a perturbation, as in fomula \ref{eq:perturbationMB}, of a 
 Hamiltonian $H$ in $\Ham$  such that the slope $a$ is big and   $\rho_0$ is small.
The non constant critical points of $\A_{H_{\delta, N}+\tilde{f}_N}$ are pairs $(\gamma_z,z)$ where $z$ is a critical point of $\tilde{f}_N$ and where $\gamma_z$ is a $\phi_z$-translation of a non constant $1$-periodic orbit $\gamma'$ of $H_{\delta}$ in $\widehat{W}$. Such a $\gamma'$   is of the form $\gamma_{\m}$ or  $\gamma_{\M}$,  located on a level
$M\times\{\rho\}, \rho\in (0,\rho_{0})$ corresponding to a periodic orbit of  $-R_{\alpha}$ of period $T=e^{\rho}h'(\rho)$. 
\begin{remark}\label{rem:action} 
	The action of this critical point $(\gamma_z,z)$ is given by \\$
		-\int_{S^1}\gamma_z^{\star}\widehat{\lambda}-\int_{S^1}(H_N+\tilde{f}_N)(\theta, \gamma_z(\theta),z)\, d\theta.$
	With our assumptions ($f$ small, $\rho_0$ small), the second term is close to zero.
	The first term is equal to $-\int_{S^1}(\gamma^{\prime} )^{\star}\widehat{\lambda}=e^{\rho}T$.
	Hence the action of this critical point is close to $T$.
	\end{remark}

We now prove {theorem} \ref{thm:computing}; $(W,\lambda)$ is a Liouville domain and  $\alpha$ a contact form on $\partial W$ such that the set of Conley-Zehnder indices of 
the set $\Per(R_{\alpha})$ of all good periodic Reeb orbits is lacunary.
We shall show that $$SH^{S^1,+}(W,\Q) = \bigoplus_{\gamma\in\Per(R_{\alpha})} \Q\langle\gamma\rangle.$$
\begin{proof} 
Let $H$ be a Hamiltonian in $\Ham$ such that the action is distinct for  $S^1$-families of orbits corresponding to Reeb orbits of different period.
This is possible by Remark \ref{rem:action}.
We consider, as mentioned above, the $S^1$-equivariant functions $H_{\delta,N}$ which are lifts of a  perturbation $H_{\delta}$ of $H$.
We use the natural identification, described in section \ref{section:specialcomplex}:
	\[
		S{\widetilde{C}}^{S^1,N,+}(H_{\delta,N},f_N) \simeq \Z[u]/\raisebox{-1ex}{$u^{N+1}$}\otimes S{{C}}^+(H_{\delta})
	\]
	and the description of $S{{C}}^+(H_{\delta})$ given by Lemma \ref{lem:orbitesMB}.
\begin{remark} \label{rmk:energy}
	The energy $E(u)=\int \norm{\partial_su}^2_{g_{J^{\theta}_{z(s)}}} dsd\theta$ of all Floer trajectories involved in the definition of the boundary operator which are linking elements $(\gamma_z,z)$ corresponding to distinct $\gamma$'s, say $\gamma_-$ and $\gamma_+$, is bounded below by some positive constant $E$ depending only on $H$.  	Indeed\footnote{this argument is borrowed from \cite{GG2}}, the result follows from the two following facts:\\
	First, $\norm{\partial_su}^2_{g_{J^{\theta}_{z(s)}}}$ is bounded above (\cite{Sal2}) since, otherwise, there would be some ``bubbling off'' which is prevented by exactness of the symplectic form.\\
	Secondly $\int_{S^1}\norm{ \partial_\theta u(s,\theta)-X_{H_{\delta,N_{z(s)}}}(u(s,\theta))}_{g_{J^{\theta}_{z(s)}}}^2d\theta$ is bounded below by an $\epsilon$ valid for any smooth loop $u(s,\cdot): S^1\rightarrow \widetilde{W} : \theta \mapsto u(s,\theta)$ with some values outside a neighborhood  of the critical orbits \cite[Exercice 1.22]{Sal}.
	This is proven by contradiction, using Arzela-Ascoli Theorem to prove that   every sequence $u_n:S^1\to\widetilde{W}$ with $\norm{\dot{u}_n(t)-X_t(u_n)}_{L^2}\mapsto 0$ has a subsequence which converges uniformly to a 1-periodic orbit of the Hamiltonian vector field.
\end{remark}

	The complex $S{\widetilde{C}}^{S^1,N,+}(H_{\delta,N},f_N)$ is filtered by the action thanks to Proposition \ref{prop:actionS1bo}.
	We take the filtration by the action, i.e. by the period ; we define $F_{p}S{\widetilde{C}}^{S^1,N,+}(H_{\delta,N},f_N), p\in\Z$ such that for every $p\in\Z$, the quotient $$F_{p+1}S{\widetilde{C}}^{S^1,N,+}(H_{\delta,N},f_N)/\raisebox{-0,5ex}{$F_pS{\widetilde{C}}^{S^1,N,+}(H_{\delta,N},f_N)$}$$
	is a union of sets
	\[
		\{1\otimes\gamma_{\M},\ldots,u^N\otimes\gamma_{\M},1\otimes\gamma_{\m},\ldots,u^N\otimes\gamma_{\m}\}
	\]
	corresponding to  underlying Reeb orbits $\gamma$ of the same period $T$.\\
	We consider the zero page of the associated spectral sequence.
	\[
		E^{0,N}_{p,q} := F_{p+1}S{\widetilde{C}}^{S^1,N,+}_{p+q}(H_{\delta,N},f_N)/\raisebox{-1ex}{$F_pS{\widetilde{C}}^{S^1,N,+}_{p+q}(H_{\delta,N},f_N)$}
	\]
	We have ``twin towers of generators'', one tower corresponding to each periodic Reeb orbit of period $T$ on $\partial W$,
	\[
		\begin{xymatrix}{
			u^N\otimes\gamma_{\M} \ar[rd]^{u^{-1}\varphi_1} & u^N\otimes\gamma_{\m} \ar[l]_{\varphi_0} \\
			\vdots & \vdots\\
			u^2\otimes\gamma_{\M} \ar[rd]^{u^{-1}\varphi_1} & u^2\otimes\gamma_{\m} \ar[l]_{\varphi_0} \\
			u\otimes\gamma_{\M} \ar[rd]^{u^{-1}\varphi_1} & u\otimes\gamma_{\m} \ar[l]_{\varphi_0} \\
			1\otimes\gamma_{\M} & 1\otimes\gamma_{\m} \ar[l]_{\varphi_0}
			}
		\end{xymatrix}
	\]
	with induced differential as in the above diagram with the notation of section \ref{section:specialcomplex}.
	The differential between two elements in distinct towers of the same period vanishes, since for any Floer trajectory involved in the differential and linking the two towers, by remarks 
	 \ref{rmk:action} and  \ref{rmk:energy},
	\begin{align*}
		E<\int \norm{\partial_su}^2_{g_{J^{\theta}_{z(s)}}} dsd\theta &\leq \A(\gamma^-,z^-)-\A(\gamma^+,z^+)\\
	\end{align*}
	and the last can be chosen to be less than $E$ (choosing small $\delta$ and $f_N$).
	
	To study any given tower, we use the explicit description of $\varphi_0$ and $\varphi_1$.
	It was first described by Bourgeois and Oancea but in their computation, they assumed transversality of contact homology.
	It was then computed without this assumption by Zhao.
	\begin{enumerate}
		\item\label{item:lem4.28}\cite[Proposition 6.2]{Zhao}, \cite[Lemma 4.28]{BOduke} Let $\gamma_{\m}$, $\gamma_{\M}$ and $H_{\delta}$ be as above.
			For $\delta$ small enough, the moduli space $\Mod(\gamma_{\m},\gamma_{\M};H_{\delta},J)/\raisebox{-0.5ex}{$\R$}$ consists of two elements;
			they have opposite signs, due to the choice of a system of coherent orientations, if and only if the underlying Reeb orbit $\gamma$ is good.
			This implies that,
			\[
				\varphi_0(\gamma_{\m})
				\begin{cases}
					0 &\textrm{if }\gamma\textrm{ is good,}\\
					\pm2\, \gamma_{\M} &\textrm{if }\gamma\textrm{ is bad}.
				\end{cases}
			\]
			Recall that a  Reeb orbit is called bad if its Conley-Zehnder index is not of the same parity as the Conley-Zehnder index of the simple Reeb orbit with same image, and
			an orbit $\gamma_H$ is bad if the underlying Reeb orbit is bad.
		\item\label{item:lem3.3}\cite[Proposition 6.2]{Zhao}, \cite[Lemma 3.3]{bo} With the same notations, the map $\varphi_1: S{{C}}^+_*(H_{\delta})\rightarrow S{{C}}^+_{*+1}(H_{\delta})$ acts by
			\[
				\varphi_1(\gamma_{\M})=
				\begin{cases}
					k_{\gamma}\gamma_{\m} &\textrm{if }\gamma\textrm{ is good,}\\
					0 &\textrm{if }\gamma\textrm{ is bad}
				\end{cases}
			\]
			where $k_{\gamma}$ is the multiplicity of the underlying Reeb orbit $\gamma$ i.e. $\gamma$ is a $k_{\gamma}$-fold cover of a simple periodic Reeb orbit.
	\end{enumerate}
	 The complex in $E^{0;N}_{p,q}$ defined by  the twin tower corresponding to a good orbit yields
	\[
		\begin{xymatrix}{
			\Q \ar[r]^0&\Q \ar[r]^{(\times k_{\gamma})}&\hdots\ar[r]^{(\times k_{\gamma})}&
			\Q \ar[r]^0&\Q \ar[r]^{(\times k_{\gamma})}&\Q \ar[r]^0&
			\Q}
		\end{xymatrix}
	\]
	and thus, in the homology $E^{1;N}_{p,q}$, it gives one copy of $\Q$ in degree $-\mu_{CZ}(\gamma)$ and one copy of $\Q$ in degree $-\mu_{CZ}(\gamma)+2N$.
	The first page is given by
	\[
		E^{1;N} = \bigoplus_{\gamma\in\Per(H_{\delta})}\Q\langle\gamma_{\M}\rangle\oplus\Q\langle u^N\otimes\gamma_{\m}\rangle.
	\]
	There are no bad orbits in the generators of the $S^1$-equivariant symplectic homology. Indeed
	the complex  in $E^{0;N}_{p,q}$ defined by  the twin tower over a bad orbit is :
		\[
		\begin{xymatrix}{
			\Q\ar[r]^{\times(\pm2)}&\Q\ar[r]^{0}&\hdots\ar[r]^{0}&\Q\ar[r]^{\times(\pm2)}&\Q\ar[r]^{0}&\Q\ar[r]^{\times(\pm2)}&\Q}
		\end{xymatrix}
	\]
	and the corresponding homology gives  $0$ in  $E^{1;N}_{p,q}$.
	
	The differential on the first page of the spectral sequence vanishes because of the lacunarity of the set of Conley-Zehnder indices; therefore, for $N$ large enough, it gives the homology :
	\[
		SH^{S^1,N,+}(H_{\delta,N}) = \bigoplus_{\gamma\in\Per(H_{\delta})}\Q\langle\gamma_{\M}\rangle\oplus\Q\langle u^N\otimes\gamma_{\m}\rangle.
	\]
	The morphism induced by a regular homotopy between two such Hamiltonians (built from standard Hamiltonians close to Morse Bott Hamiltonians) respects the filtration, thanks to proposition \ref{prop:actionS1bo}.
	We can therefore take the direct limit on  the pages over those Hamiltonians which form a cofinal family.
	The inclusion  $S^{2N+1}\hookrightarrow S^{2N+3}$  induces a map
	\[
		E^{1;N} = \bigoplus_{\gamma\in\Per(R_{\alpha})}\Q\langle\gamma_{\M}\rangle\oplus\Q\langle u^N\otimes\gamma_{\m}\rangle \rightarrow  E^{1;N+1}.
	\]
	which is the identity on the first factor and zero on the second factor.	
	Taking the direct limit over the inclusion $S^{2N+1}\hookrightarrow S^{2N+3}$ we have\\

		 ${\displaystyle{\qquad\qquad\qquad SH^{S^1,+}(W;\Q)= \lim_{\substack{\longrightarrow\\ N}} E^{1;N} = \bigoplus_{\gamma\in\Per(R_{\alpha})}\Q\langle\gamma\rangle.}}
	$
	 \end{proof}
\begin{remark}
	Stricto sensu, in the proof of the above Theorem, we have assumed that the orbits are contractible.
	Nonetheless Theorem \ref{thm:computing} is true after extending the definition of $SH^{S^1,+}(H)$ to all $1$-periodic orbits of $H$.
	To deal with non contractible orbits, one chooses for any free homotopy class of loops $a$, a representative $l_a$ and one chooses a trivialisation of the tangent space along that curve.
For the free homotopy class of  a contractible loop, $l_0$ is chosen to be constant loop with constant trivialisation.
One ask moreover that $l_{a^{-1}}$ is $l_a$ in the reverse order and with the corresponding trivialisation.
The action functional induced by a Hamiltonian $H$ becomes
\[
	\A(\gamma) := -\int_{[0,1]\times S^1}u^{\star}\w-\int_{S^1}H(\theta,\gamma(\theta))d\theta
\]
where $u:[0,1]\times S^1\rightarrow W$ is a homotopy from $l_a$ to $\gamma$.
For any loop $\gamma$ belonging to the free homotopy class $a$, one chooses a homotopy $u:[0,1]\times S^1\rightarrow W$ from $l_a$ to $\gamma$ and one considers the trivialisation of $TW$ on $\gamma$ induced by $u$ and
by the choice of the trivialisation along $l_a$.
Let us observe that any Floer trajectory can only link two orbits in the same free homotopy class and as before, the action decreases along Floer trajectories.
As before, the Floer complex is generated by the $1$-periodic orbits  of $H$ graded by minus their Conley-Zehnder index.
The differential ``counts'' Floer trajectories between two orbits whose difference of grading is $1$.
The positive version of symplectic homology is defined as before since the set of critical points of $H$ is still a subcomplex : Floer trajectories can only link a critical point to a contractible orbit.
All the results stated above extend to this framework.
\end{remark}

\begin{cor}\label{cor:complexSHS1+}
	The only generators that may appear in the positive $S^1$-equivariant homology are of the form $u^0\otimes\gamma_{\M}$ with $\gamma_{\M}$ a good orbit.
\end{cor}
\begin{cor}
	The number of good periodic Reeb orbits of periods $\leq T$ is bounded below by the rank of the positive $S^1$-equivariant symplectic homology of action $\leq T$.
\end{cor}

\section{Structural properties of symplectic homology}\label{chapter:properties}

\subsection{Transfer morphism for symplectic homology}\label{section:transfer}
In this section, we prove that symplectic homology, positive symplectic homology, $S^1$-equivariant symplectic homology and positive $S^1$-equivariant symplectic homology are functors (reversing the arrows) defined on the category where objects are Liouville domains, and morphisms are embeddings.
Precisely, we  construct  a  morphism between the ($S^1$- equivariant positive) symplectic homologies when one Liouville domain is embedded in another one, and we show that those morphisms compose nicely.
Such a morphism, called a transfer morphism, has been studied by Viterbo \cite{V} in the case of the symplectic homology. We adapt his construction to extend it to all the variants of the symplectic homology
considered above.
We consider a {{Liouville domain}} $(W,\lambda)$ and
its completion $\widehat{W}=W\cup (\partial W\times \R^+)$ 
built from the flow of the Liouville vector field $X$ as in definition \ref{def:setup}.
We denote by $\widehat{\lambda}$ the $1$-form on $\widehat{W}$ defined by $\lambda$ on $W$ and by $e^\rho\alpha$ on $\partial W\times \R^+$ with $\alpha:= \lambda_{\vert_{\partial V}}$.
We  denote by $SH^\dagger (W,\lambda)$ its symplectic homology $SH^\dagger (W,d\lambda,X)$.
\begin{definition}
	Let $(V,\lambda_V)$ and $(W,\lambda_W)$ be two Liouville domains.
	A \emph{Liouville embedding} $j:(V,\lambda_V)\rightarrow (W,\lambda_W)$ is a symplectic embedding $j:V\rightarrow W$ with $V$ and $W$ of codimension $0$ such that  
	$j^{\star}\lambda_W=\lambda_V$.
(One can consider, more generally, a symplectic  embedding $j$ of codimension $0$ such that  $\lambda_W$ coincides in a neighbourhood of $j(\partial V)$ in $W$ with $\widehat{\lambda_V} + df$.)\end{definition}
To construct  transfer morphisms, we introduce a special class of Hamiltonians $\Hs_{stair}(V,W)$ and we use, as in \cite{V},  increasing homotopies between $H_1 : S^1\times \widehat{W}\rightarrow\R \in \Hstd(W)$ and an $H_2:S^1\times \widehat{W}\rightarrow\R$ in  $\Hs_{stair}(V,W)$.

\begin{definition}\label{def:Hstair}
We fix a neighbourhood $U$ of $\partial V$ in $W\setminus \mathring{V}$
so that $(U,\omega_W)$ is symplectomorphic to $\bigl(\partial V\times[0,\delta], d(e^{\rho}\alpha_V)\bigr)$.\\
A Hamiltonian $H_{2} : S^1\times \widehat{W}\rightarrow\R$  is in $\Hs_{stair}(V,W)$ if and only if
	\begin{itemize}
		\item on $S^{1}\times V$, $H_2$ is negative and $C^{2}$-small  ;
		\item on $S^1\times U\cong S^1\times \partial V\times[0,\delta]$, with $\rho$ the last coordinate, $H_2$ is of the following form
		\begin{itemize}
			\item there exists $0<\rho_{0}\ll\delta $ such that $H_2(\theta, p, \rho) = \beta e^{\rho} + \beta' \textrm{ for } \rho_{0}\leq\rho\leq \delta-\rho_{0}$, with
				$0<\beta\notin\Spec(\partial V,\alpha)\cup\Spec(\partial W,\alpha)$  and $\beta'\in\R$;
			\item $H_2(\theta, p, \rho)$ is $C^{2}$-close on $S^{1}\times \partial V\times[0,\rho_{0}]$ to a  convex increasing function of $e^\rho$ which is independent of $\theta$ and $p$;
			\item $H_2(\theta, p, \rho)$ is $C^{2}$-close on $S^{1}\times \partial V\times[\delta-\rho_{0},\delta]$ to a concave increasing function of $e^\rho$ which is independent of $\theta$ and $p$;
		\end{itemize}
		\item on $S^{1}\times W\setminus(V\cup U)$, $H_2$ is $C^{2}$-close to a constant ;
		\item on $S^1\times  \partial W\times[0,+\infty[$, with $\rho'$ the $\R^+$ coordinate on $\partial W\times\R^+$, $H_2$ is of the following form
		\begin{itemize}
			\item there exists $\rho'_{1} > 0$  such that $H_2(\theta, p, \rho') = \mu e^{\rho'} + \mu' \textrm{ for } \rho'\geq\rho'_{1}$, with
				$0<\mu\notin\Spec(\partial V,\alpha)\cup\Spec(\partial W,\alpha)$, $\mu<\frac{\beta(e^{\delta}-1)}{e^\delta}$, $\mu'\in\R$;
			\item $H_2(\theta, p, \rho')$ is $C^{2}$-close on $S^{1}\times \partial W\times]0,\rho'_{1}]$ to a concave increasing function of $e^{\rho'}$  which is independent of $\theta$ and $p$;
		\end{itemize}
		\item all $1$-periodic orbits of $X^{\theta}_{H_2}$ are non-degenerate, i.e the Poincar\'e return map has no eigenvalue equal to $1$.
	\end{itemize}
\end{definition}
\noindent A representation of $H_2$ is given in Figure \ref{hamiltoniens}.
\begin{figure}
	\hspace{3cm}
	\def\svgwidth{1\textwidth}
	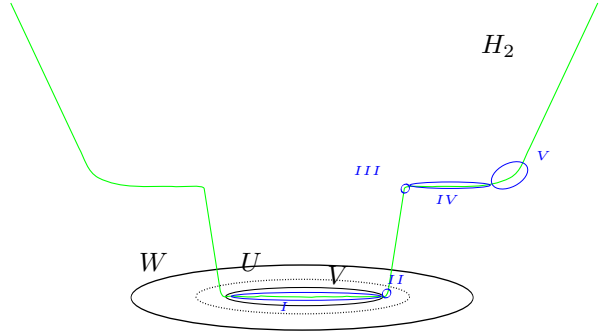
	\caption{Example of $H_2$ on $\widehat{W}$}\label{hamiltoniens}
\end{figure}

The $1$-periodic orbits of $H_2$ lie either in the interior $\mathring{V}$ (which we call region \rom{1}), either in $\partial V\times[0,\rho_0]$ (region \rom{2}),
either in $\partial V\times[\delta-\rho_0,\delta]$ (region \rom{3}), either in $W\setminus(V\cup U)$ (region \rom{4}) or in $\partial W\times[0,\rho_1]$ (region \rom{5}).
We consider their action (using the obvious fact that
	if $H$ and $\widetilde{H}$ are two $C^2$-close Hamiltonians and if $\gamma\in\Per(H)$ and $\widetilde{\gamma}\in\Per(\widetilde{H})$ are $C^2$-close,
	then $\A(\gamma) \textrm{ is close to } \A(\widetilde{\gamma}).$)
\begin{itemize}
	\item[\textbf{\rom{1}}] {In region \rom{1}}, there are only critical points so the action of the critical point $q$ is non negative and small ($<\epsilon$).
	\item[\textbf{\rom{2}}] {In region \rom{2}}, $H_2$ is $C^2$-close to a  convex function $H=h(r)$ (with $r=e^{\rho}$); since
	$\omega_W =d(r\alpha_V)$, we have $X_H=-h'(r)R_{\alpha_V}$ where $R_{\alpha_V}$ is the Reeb vector field on $\partial V$ associated to the contact form $\alpha_V={\lambda_V}_{|_{\partial V}}$.
		An orbit of $X_H$ lies on a constant level for $r$ and its action is given by:
	\begin{eqnarray*}
			\A(\gamma) &=& -\int_{S^1}\gamma^{\star}(r\alpha_V)-\int_{S^1}H\bigl(\gamma(\theta)\bigr)d\theta=-\int_{S^1}r\alpha_V\bigl(-h'(r)R_{\alpha_V}\bigr)-h(r)\\
			&=&h'(r)r-h(r).
	\end{eqnarray*}
		Since $\rho_0$ is small we have $e^{\rho_0}\sim 1$ and $ h(e^{\rho_0})\sim 0$, so the actions of $1$-periodic orbits of $H_2$ in this region are close to the periods of closed orbits of the Reeb vector field on the boundary of $V$ of periods $T<\beta$ and they are greater than $\epsilon$.
	\item[\textbf{\rom{3}}] {In region \rom{3}}, the computation is similar to the case of region \rom{2}:
		$ \A(\gamma_{H_2})$ is equal to $h'(r)r-h(r)$ which  is less than  {$e^{\delta}\beta-(e^{\delta}\beta-\beta)=\beta$}\footnote{The author is grateful to Oleg Lazarev who pointed out a mistake in an earlier computation}.
	\item[\textbf{\rom{4}}] {In region \rom{4}}, there are only critical points so the action of the critical point $q$ is given by  $-H_2(q)$ which is close to {$- (e^{\delta}\beta-\beta)$}.
	\item[\textbf{\rom{5}}] {In region \rom{5}}, the computation of the action is similar to the case of region \rom{2}:
		$\A(\gamma)$ is close to $h'(r)r-h(r)$  with $ r=e^{\rho'}$.
		Observe that here the $1$-periodic orbits are close to $1$-periodic orbits of $-h'(r)R_{\alpha_W}$ where now $R_{\alpha_W}$ is the Reeb vector field on $\partial W$.
		The action of any $1$-periodic orbit of $H_2$ in this region is close to $e^{\rho'} T'-h(e^{\rho'})$ where $T'$ is the period of a closed orbit of the Reeb vector field on the boundary of $W$ with $T'<\mu$ and {thus the action is $<\mu e^{\delta}-\beta(e^{\delta}-1)<0$}.

\end{itemize}
So, for nice parameters (for instance $\rho'_1<\delta$), we have
$$
	\A(\rom{4})<\A(\rom{5})<0<\A(\rom{1})<\epsilon<\A(\rom{2}).
$$
{There are no Floer trajectories from {\rom{3}} to \rom{1} or \rom{2} by \cite[Lemma 2.3]{CO}.}
We denote by $C^{\rom{4}, \rom{3}, \rom{5}, \rom{1}}(H_2,J)$ the subcomplex of the Floer complex for $H_2$ generated by critical orbits lying in regions \rom{4}, \rom{3}, \rom{5}, and \rom{1} and by  $C^{\rom{4}, \rom{3}, \rom{5}}(H_2,J)$  the subcomplex of the Floer complex for $H_2$ generated by critical orbits lying in regions \rom{4}, \rom{3} and  \rom{5}.
We have the identifications:
\begin{eqnarray*}
	C^{\rom{1},\rom{2}}(H_2,J) &=& C^{\rom{4}, \rom{3}, \rom{5}, \rom{1}, \rom{2}}(H_2,J)/\raisebox{-1ex}{$C^{\rom{4}, \rom{3}, \rom{5}}(H_2,J)$}\\
	C^{\rom{2}}(H_2,J) &=& C^{\rom{4}, \rom{3}, \rom{5}, \rom{1}, \rom{2}}(H_2,J)/\raisebox{-1ex}{$C^{\rom{4}, \rom{3}, \rom{5}, \rom{1}}(H_2,J)$}\\
\end{eqnarray*}
We have subcomplexes using the fact that the action decreases along Floer trajectories, {and using \cite[Lemma 2.3]{CO}.}
The Floer differential passes to the quotient where we still denote it $\partial$.
Remark that the function $H_2$ is not  in $\Hstd(V)$. We want to relate the homology of $\bigl(C^{\rom{1},\rom{2}}(H_2,J),\partial\bigr)$   to the homology of a function in $\Hstd(V)$.
\begin{definition}
Let $H_2\in\Hs_{stair}(V,W)$; we denote by   $\beta$ the slope of the linear part close to $\partial V$, as in Definition \ref{def:Hstair}. We  associate to $H_2$ the  function denoted  $H=\iota_V(H_2) \in \Hstd(V)$, defined on  $S^1\times \widehat{V}$,
 which coincides with $H_2$ on $V\cup (\partial V\times [0,\delta-\rho_0])$ and which is linear with slope $\beta$ ``further''  in the completion: $H(\theta, e^\rho)=\beta \rho+ \beta'$ for all $\rho\geq \delta-\rho_0$.
\end{definition}
\begin{proposition}\label{prop:homCII}
	Let $H_2$ be an function in $\Hs_{stair}$ and let $H=\iota_V H_2$ be the associated function in  $\Hstd(V)$ as defined above. 	We assume furthermore that the Hamiltonians are generic in the sense that the homologies are well-defined for a good choice of $J$'s.
	Then
	$$
		H\bigl(C^{\rom{1},\rom{2}}(H_2,J),\partial\bigr)=H\bigl(SC(H,J)\bigr)\, \textrm{and}\, H\bigl(C^{\rom{2}}(H_2,J),\partial\bigr)= H\bigl(SC^{+}(H,J)\bigr).
	$$
\end{proposition}
\begin{proof}
	We need to check that there is no Floer trajectory $u:\R\times S^1\rightarrow \widehat{W}$ going from an orbit in $C^{\rom{1},\rom{2}}$ (resp. $C^{\rom{2}}$) to an orbit in $C^{\rom{1},\rom{2}}$ (resp. $C^{\rom{2}}$)
 with points in $\widehat{W}\setminus (U\cup V)$. We prove it by contradiction,
	as a direct application of Abouzaid maximum principle which we prove below as theorem \ref{thm:abouzaid}. Assume that
	$ u:\R\times S^1\rightarrow \widehat{W}$ is a Floer trajectory whose image intersects $\widehat{W}\setminus (U\cup V)$.
	We consider the intersection of the image with a slice $\partial V\times\{ \rho\} $ for any $\rho_0< \rho<\delta-\rho_0$
	and we choose a regular value $\rho_0 + \epsilon$ of $\rho \circ u$. 
	The manifold $W':=\widehat{W}\setminus ( V\cup( \partial V\times [0,\rho_0 + \epsilon[))$ is symplectic with
	 contact type with 
         boundary $ \partial V \times \{\rho_0 + \epsilon\} $ and Liouville vector field pointing inwards. Let  $S$ be the inverse image of $W'$ under the map 
         $u$;  it is a compact Riemann surface with boundary ; the complex structure $j$ is the restriction to $S$ of the complex structure $j$ on the cylinder defined by $j(\partial_s)=\partial_\theta$. We define 
         $\beta$ to be the restriction of  $d\theta$ to $S$. The fact that $u$ is a Floer trajectory is equivalent to
         $(du - X_{H}\otimes\beta)^{0,1}:=\half \left((du - X_{H}\otimes\beta) +J(du - X_{H}\otimes\beta)j\right)=0$,
         where $du$ is the differential of the map $u$ viewed as a section of $T^{\star}S\otimes u^{\star}TW'$.
	Then part a of theorem \ref{thm:abouzaid}, which is slight generalisation of a theorem of Abouzaid, concludes.
\end{proof}
\begin{theorem}[Abouzaid, \cite{Rit}]\label{thm:abouzaid}
	Let $(W',\w'=d\lambda')$ be an exact  symplectic manifold with contact type boundary $\partial W'$, such that the Liouville vector field points inwards.
	Let $\rho$ be the coordinate near $\partial W'$ defined by the flow of the Liouville vector field starting from the boundary and let $r:=e^\rho$; near the boundary the symplectic form writes   $\w'=d(r\alpha)$
	with $\alpha$ the contact form on $\partial W'$ given by the restriction of $\lambda'$.
	Let $J$ be a compatible almost complex structure such that $J^*\lambda'= dr$ on the boundary.\\
	a) Let  $H:W'\rightarrow \R$ be non negative, and such  that  $H=h(r)$ where $h$ is a convex increasing function near the boundary.
	Let $S$ be a compact Riemann surface with boundary and let $\beta$ be a 1-form such that  $d\beta \geq 0$.
	Then any solution $u:S\rightarrow W'$ of $(du - X_{H}\otimes\beta)^{0,1}=0$ with $u(\partial S) \subset \partial W'$
	is entirely contained in $\partial W'$.\\
	b) Let  $H:\R\times S^1\times W'\rightarrow \R$ be an increasing homotopy, such  that  $H(s,\theta,p,\rho)$ $=H_s^\theta (p,\rho)= h_s(r) $ where $h_s$ are convex increasing functions near the boundary. Let $S$ be a compact Riemann surface with boundary embedded in the cylinder {($\R\times S^1$ with the standard structure)}.
Then any solution $u:S\rightarrow W'$ of $(du - X_{H_s}\otimes d\theta)^{0,1}=0$ with $u(\partial S) \subset \partial W'$
is entirely contained in $\partial W'$.
\end{theorem}
\begin{proof} Proof of part a. The energy of a map $u: S\rightarrow W'$ is defined as 
$
E(u):=\half \int_S \Vert du-X_H\otimes \beta\Vert^2 vol_S
$
where $du$ is viewed as a section of $T^{\star}S\otimes u^{\star}TW'$. If  $s+it$ is a local holomorphic coordinate on $S$, so that $j(\partial_s)=\partial_t$ and $vol_S=ds\wedge dt$ we have
 \begin{align*}
\half\Vert du-X_H\otimes \beta\Vert^2 vol_S&=\w'\left(\partial_su -X_H\beta(\partial_s),\partial_tu -X_H\beta(\partial_t)\right)ds\wedge dt\\
& =\left( \w' (\partial_su, \partial_tu)-dH(\partial_tu)\beta(\partial_s) +dH(\partial_su)\beta(\partial_t)\right) ds\wedge dt\\
&= u^{\star}\w'+u^{\star}(dH)\wedge \beta.
\end{align*}
It is obviously non negative for any path.
Since $d(u^{\star}H\beta) = u^{\star}(dH)\wedge\beta + \underbrace{u^{\star}Hd\beta}_{\geq 0}$, we have
\begin{align*}
		E(u) &= \int_{S} u^{\star}d\lambda' + u^{\star}(dH)\wedge\beta
		\leq \int_{S} d(u^{\star}\lambda') + d(u^{\star}H\beta)   
		\leq \int_{\partial S} u^{\star}\lambda' - \lambda'(X_{H})\beta\\
		&\qquad \textrm{since }  H =h(r) \leq rh'(r)=
			r\alpha \bigl(h'(r)R_{\alpha}\bigr) = -\lambda'(X_{H}) \textrm{ on } u(\partial S)\subset\partial V\,  \\
		&= \int_{\partial S} \lambda'(du-X_{H}\otimes\beta)\\
		&= \int_{\partial S} -\lambda' J(du-X_{H}\otimes\beta)j 
		\qquad \textrm{since } (du - X_{H}\otimes\beta)^{0,1}=0\\
		&=  \int_{\partial S} -dr(du-X_{H}\otimes\beta)j \quad \textrm{ since  }J^{\star}\lambda'= dr \textrm{  along }u(\partial S)\subset\partial W'\\
		&= \int_{\partial S} -dr\,  du\,  j
		\qquad \textrm{ since  } dr\textrm{ vanishes on }X_{H}\textrm{ on }u(\partial S)\subset\partial W'.
	\end{align*}
	Let $\nu$ be the outward normal direction along $\partial S$.
	Then $(\nu,j\nu)$ is an oriented frame, so $\partial S$ is oriented by $j\nu$.
	Now $dr(du)j(j\nu)= d(r\circ u)(-\nu)\geq0$ since in the inward direction, $-\nu$, $r\circ u$ can only increase because $r$ is minimum on  $\partial W'$.
	So $E(u) \leq 0$
	hence $E(u)=0$. This implies that $du-X_H\otimes \beta=0$ which shows that the image of $du$ is in the 
	span of $X_H$ which is the span of $R_\alpha\in T\partial W'$ on $\partial W'$.
	Hence the image of $u$ is entirely in contained in $\partial W'$.\\
Proof of part b. The proof starts as above. The energy of $u$ is non negative and given by
\[
E(u):=\half \int_S \Vert du-X_{H_s}\otimes d\theta\Vert^2 vol_S= \int_S u^{\star}\w'+u^{\star}(dH^\theta_s)\wedge d\theta.
\]
We have  $u^{\star}(dH^\theta_s)\wedge d\theta = d(u^{'\star}H)\wedge d\theta- \underbrace{u^{\star}\partial_sH^\theta_s ds\wedge d\theta}_{\geq 0}$, for $u':S\rightarrow \R\times S^1\times W' $ which maps an element $(\theta,s)\in S$ to the element $(s,\theta, u'(\theta,s))$.
Hence 
 \begin{align*}
		E(u) &= \int_{S} u^{\star}d\lambda' + u^{\star}(dH)\wedge d\theta\\
		&\leq \int_{S} d(u^{\star}\lambda') + d(u^{'\star}H d\theta)\\
		&\leq \int_{\partial S} u^{\star}\lambda' - \lambda'(X_{H_s})d\theta \textrm{ using Stokes's theorem and}\\
		&\qquad H =h_s(r )\leq r\alpha \bigl(h_s'(r)R_{\alpha}\bigr) = -\lambda'(X_{H_s}) \textrm{ on } u(\partial S)\subset\partial V \\
		&= \int_{\partial S} \lambda'(du-X_{H_s}\otimes d\theta)\\
\end{align*}
and the proof proceeds as in part a.
\end{proof}

For any element $H_1\in\Hstd(W)$, one can consider an element in $H_2\in\Hs_{stair}(V,W)$ such that $H_1$ and $H_2$ coincide ``far in the completion'', i.e. on $\partial W\times [\rho'_2,+\infty[\subset \widehat{W}$. Let $H=\iota_V(H_2) \in \Hstd(V)$. We want to build a morphism from the homology defined by  $H_1$ to the  homology 
defined by $H$. We shall first construct a morphism in the homology defined by $H_2$.
With $H_1\in\Hstd(W)$ and $H_2\in\Hs_{stair}(V,W)$ as above, we can consider an increasing  homotopy $H_s, s\in \R,$ between $H_1$ and $H_2$, 
i.e $\frac{d}{ds}H_s\geq0$,  with the property that there exists $s_0$ such that $H_s\equiv H_1$ for $s\leq-s_0$ and $H_s\equiv H_2$ for $s\geq s_0$.
We  define a morphism $SC(H_1,J_1)\rightarrow SC(H_2,J_2)$ by counting Floer trajectories for the homotopy.
Denote by $\Mod(\gamma_1,\gamma_2,H_s,J_s)$ the space of Floer trajectories from $\gamma_1$ to $\gamma_2$ i.e maps $u:\R\times S^1\rightarrow \widehat{W}$
such that:
\begin{equation}\label{eq:Floerhomotopy}
	 \partial_su + J_{s}^{\theta}\circ u(\partial_{\theta}u - X_{H_{s}}^{\theta}\circ u)=0
\end{equation}
with $ \lim_{s\rightarrow-\infty}u(s,\cdot) = \gamma_1(\cdot)$ and $\lim_{s\rightarrow\infty}u(s,\cdot) = \gamma_2(\cdot).$
It is proven in \cite{O,FHS} that for a generic choice of the pair $(H_s,J_s)$, the spaces $\Mod(\gamma_1,\gamma_2,H_s,J_s)$ are manifolds of dimension $\mu_{CZ}(\gamma_2)-\mu_{CZ}(\gamma_1)$
for any $\gamma_1$ in $\Per(H_1)$ and $\gamma_2$ in $\Per(H_2)$.
Let us observe that there is no general $\R$-action on this space.
The homotopy $H_s$ gives rise to a  { morphism}
\[
	\phi_{H_s}: SC(H_1,J_1)\rightarrow SC(H_2,J_2) : 
	\gamma_1\mapsto \sum_{\substack{\gamma_2\in\Per(H_2)\\\mu_{CZ}(\gamma_2)=\mu_{CZ}(\gamma_1)}}\#\Mod(\gamma_1,\gamma_2,H_s,J_s)\gamma_2
\]
where the count involves, as always, signs. The study  of the boundary of a space of Floer trajectories $\Mod(\gamma_1,\gamma_2,H_s,J_s)$
for  $\gamma_1\in\Per(H_1)$ and $\gamma_2\in\Per(H_2)$ such that  $\mu_{CZ}(\gamma_1)=\mu_{CZ}(\gamma_2)+1$ shows that	the morphism $\phi_{H_s}$ is a chain map,
hence induces a morphism in homology, still denoted 
$
	\phi_{H_s}: SH(H_1,J)\rightarrow SH(H_2,J).
$
The homotopy of homotopies theorem shows that $\phi_{H_s}$ is independent of the choice of the homotopy; 
hence we denote it by $\phi_{H_1,H_2}$. 
\begin{definition} 
Given an element $H_1$ in $\Hstd(W)$, consider  an element $H_2\in\Hs_{stair}(V,W)$ such that $H_1$ and $H_2$ coincide ``far in the completion'', and let $H=\iota_V(H_2) \in \Hstd(V)$. We define the {\emph{transfer morphism}}
$$
SH( H_1,J)\rightarrow SH( H,J')=SH(H_2,J)/\raisebox{-1ex}{$SH^{\leq -\eta}(H_2,J)$}=H\bigl(C^{\rom{1},\rom{2}}(H_2,J),\partial\bigr)
$$
which is the composition of $\phi_{H_1,H_2}$ followed by the natural projection.\\
The action decreases along  Floer trajectories, so this maps $SH^{\leq \epsilon}( H_1,J)$ 
to 
$$
SH^+( H,J')=SH^{\leq \epsilon}(H_2,J),\partial)/\raisebox{-1ex}{$SH^{\leq -\eta}(H_2,J),\partial)$}=H\bigl(C^{\rom{2}}(H_2,J),\partial\bigr)
$$ and induces a  {\emph{transfer morphism}} for the positive 
homology
$$
SH^+( H_1,J)\rightarrow SH^+( H,J')=H\bigl(C^{\rom{2}}(H_2,J),\partial\bigr).
$$
With our identification, the map is obtained by counting  solutions of equation \eqref{eq:Floerhomotopy}
going from a $1$-periodic orbit of $X_{H_1}$ to a $1$-periodic orbit of $X_{H_2}$ lying in region $\rom{1}$ or $\rom{2}$.
\end{definition}
The  homotopy of homotopies theorem  shows  that the map does not depend
on the choice of stair function $H_2$ such that $\iota_VH_2=H$ and
such that $H_1$ and $H_2$ coincide far in the completion; we shall denote it $\phi^H_{H_1}$. It also shows that the map $\phi_{H_1,H_2}$ commutes with continuation, i.e if $\rho_1: SH(H_1)\rightarrow SH(H'_1)$ is a continuation for $H_1$ and
$\rho_2 : SH(H_2)\rightarrow SH(H'_2)$ is a continuation for $H_2$ then
$\phi_{H'_1,H'_2}\circ\rho_1=\rho_2\circ\phi_{H_1,H_2}.$
\begin{proposition}\label{prop:transfer2}
	The transfer map $\phi^H_{H_1}: SH( H_1,J)\rightarrow SH( H,J')$ commutes with continuations.
\end{proposition}
\begin{proof}
To show this,  we still have to show that a continuation map built in $W$ from  $SH(H_2, J)$ to $SH(H'_2,J')$, defined by an increasing  homotopy  $H_s:S^1\times \widehat{W}\rightarrow \R$, induces a continuation map in $V$ from $SH(H=\iota_V(H_2),J)$ to $SH(H'=\iota_V(H_2'),J')$. 
For this,  it is enough to check that there is no Floer trajectory corresponding to the homotopy, i.e. $u:\R\times S^1\rightarrow \widehat{W}$ solution of 
 \eqref{eq:Floerhomotopy} going from an orbit in $C^{\rom{1},\rom{2}}(H_2, J)$ (resp. $C^{\rom{2}}(H_2, J)$) to an orbit in $C^{\rom{1},\rom{2}}(H'_2,J')$ (resp. $C^{\rom{2}}(H'_2,J')$)
 with points in $\widehat{W}\setminus (U\cup V)$. We prove it by contradiction, proceeding as in the proof
 of Proposition \ref{prop:homCII}, using the generalized Abouzaid maximum principle proven in part b of  Theorem \ref{thm:abouzaid}. Assume that
	$ u:\R\times S^1\rightarrow \widehat{W}$ is a Floer trajectory whose image intersects $\widehat{W}\setminus (U\cup V)$.
	We consider the intersection of the image with a slice $\partial V\times\{ \rho\} $ for any $\rho_0< \rho<\delta-\rho_0$
	and we choose a regular value $\rho_0 + \epsilon$ of $\rho \circ u$. 
	The manifold $W':=\widehat{W}\setminus ( V\cup( \partial V\times [0,\rho_0 + \epsilon[))$ is symplectic with
	 contact type with 
         boundary $ \partial V \times \{\rho_0 + \epsilon\} $ and the  Liouville vector field pointing inwards. Let  $S$ be the inverse image of $W'$ under the map 
         $u$;  it is a compact Riemann surface embedded in the cylinder with boundary ; the complex structure $j$ is the restriction to $S$ of the complex structure $j$ on the cylinder defined by $j(\partial_s)=\partial_\theta$.  The fact that $u$ is a Floer trajectory is equivalent to
         $(du - X_{H_s}\otimes d\theta)^{0,1}:=\half \left((du - X_{H_s}\otimes d\theta) +J(du - X_{H_s}\otimes d\theta)j\right)=0$,
         where $du$ is the differential of the map $u$ viewed as a section of $T^*S\otimes u^*TW'$.
	Then part b of Theorem \ref{thm:abouzaid} concludes.
	\end{proof}
\begin{cor}
	The  maps $\{\, \phi_{H_1}^H\,\}$ induce a transfer map:
$
\phi_{W,V} : SH(W,\lambda_W)\rightarrow SH(V,\lambda_V)
$
and, on the quotient, the morphism
$
\phi^+=\phi^+_{W,V}: SH^+(W,\lambda_W)\rightarrow SH^+(V,\lambda_V).
$
\end{cor}
\begin{theorem}[Composition]\label{thm:compos}
	Let $(V_1,\lambda_{V_1}) \subseteq (V_2,\lambda_{V_2})\subseteq (V_3,\lambda_{V_3})$ be Liouville domains with Liouville embeddings.
	Then the following diagram commutes:
	\begin{equation}\label{diagcommutatifSH+}
		\xymatrix{
			SH^{+}(V_3,\lambda_{V_3}) \ar[r]^{\phi^+_{V_3,V_2}} \ar@/_2pc/[rr]^{\phi^+_{V_3,V_1}} & SH^{+}(V_2,\lambda_{V_2}) \ar[r]^{\phi^+_{V_2,V_1}} & SH^{+}(V_1,\lambda_{V_1})
		}
	\end{equation}
\end{theorem}
\begin{proof}The proof results from the comparison of  a count  of Floer trajectories. On one hand, one counts Floer trajectories corresponding to an increasing homotopy $H_{13}$, going from a $1$-periodic orbit of $X_{H_1}$ for an admissible Hamiltonian $H_1$ on $S^1\times \widehat{V_3}$ to the $C^{\rom{2},\rom{1}}$ part of a stair Hamiltonian $H_3$
with two ``steps''.  On the other hand, one  counts trajectories relative to the composition of two increasing homotopies, $H_{12}$ going from $H_1$ to $H_2$ (a stair Hamiltonian with one step) and $H_{23}$ going from $H_2$ to $H_3$. The property is a consequence  of the composition of homotopies.
\end{proof}

\subsection{Transfer morphism for $S^1$-equivariant symplectic homology}\label{sect:transferS1}

We extend the definition of the transfer morphisms  of the
previous section to  $S^1$-equivariant and positive $S^1$-equivariant symplectic homology. 
We consider two embedded Liouville domains $(V,\lambda_V)\subset (W,\lambda_W)$ 
and we want to define a morphism $SH^{S^1}(W,\lambda_W)\rightarrow SH^{S^1}(V,\lambda_V)$.
We start with  autonomous Hamiltonians $H$ in $\Hstd$, we do  small Morse Bott type deformations $H_{\delta}$   and 
then lift those to $S^1$-equivariant functions $H^{N}_{\delta}$.
In this setting, the $S^1$-equivariant symplectic homology can be computed by the simplified complex
as described in sections  \ref{section:specialcomplex} and \ref{section:computS1} :
\[
	S{\widehat{C}}_*^{S^1}(H_{\delta}) := \Z[u]\otimes_{\Z}SC_*(H_\delta)
\]
with differential ${\widehat{\partial}}^{S^1} = \varphi_0+u^{-1}\varphi_1+u^{-2}\varphi_2+\ldots$
where the maps $\varphi_j$ counts Floer trajectories for parametrized Hamiltonians
 going from $S^1 \cdot(\gamma^-,z_j)$ to $S^1 \cdot(\gamma^+,z_0)$ with $z_j$ the critical point of $f$ of index $-2j$.

The action of the element represented by $u^k\otimes \gamma$
is very close to the action of $\gamma$.
To define transfer morphisms, we start with an  autonomous Hamiltonian $H_1$ in $\Hstd(W)$ and an autonomous  $H_2$ in $\Hs_{stair}(W)$, and we do  small Morse Bott type deformations $H_{1\delta}$ et $H_{2\delta}$.
 We define as in the previous section the
subcomplex $ \Z[u]\otimes_{\Z}(C^{\rom{3},\rom{4},\rom{5}}(H_{2\delta}))$ corresponding to points with negative action and elements in region $\rom{3}$,
and we identify the quotient $ \Z[u]\otimes_{\Z}SC_*(H_{2\delta})/\raisebox{-0.5ex}{$ \Z[u]\otimes_{\Z}(C^{\rom{3},\rom{4},\rom{5}}(H_{2\delta}))$}$ to $ \Z[u]\otimes_{\Z}C^{\rom{1},\rom{2}}(H_{2\delta})$.
We consider the Hamiltonian $\iota_V H_{2\delta}$ in  $\Hstd(W)$.
\begin{proposition}
For $\delta$ small enough, the $S^1$ equivariant homology of the quotients coincide with the $S^1$ equivariant homology of the small domain:
$$
		H\bigl( \Z[u]\otimes_{\Z}C^{\rom{1},\rom{2}}(H_{2\delta}),\partial\bigr)=H\bigl(SC( \Z[u]\otimes_{\Z}SC(\iota_V H_{2\delta}))\bigr)$$
		$$
	H\bigl( \Z[u]\otimes_{\Z}C^{\rom{2}}(H_{2\delta}),\partial\bigr)=H\bigl(SC^+( \Z[u]\otimes_{\Z}SC(\iota_V H_{2\delta}))\bigr)	$$
\end{proposition}
\begin{proof}
What remains to be checked  is again  there is no parametrized Floer trajectory $u:\R\times S^1\rightarrow \widehat{W}$ going from an orbit in $C^{\rom{1},\rom{2}}(H_{2\delta})$ to an orbit in
$C^{\rom{1},\rom{2}}(H_{2\delta})$ with points in $\widehat{W}\setminus (U\cup V)${; this is due to the decomposition of ${\widehat{\partial}}^{S^1}$ mentioned above.}
This is proven by contradiction. If there was a parametrized trajectory going from an orbit in $C^{\rom{1},\rom{2}}(H_{2\delta})$ 
to an orbit in $C^{\rom{1},\rom{2}}(H_{2\delta})$ with points in $\widehat{W}\setminus (U\cup V)$ for all $\delta$'s,
then, by a theorem of Bourgeois and Oancea \cite[Proposition 4.7]{BOduke}, there would be such a broken trajectory for the autonomous Hamiltonian and we have proven in Proposition \ref{prop:homCII} that this can not exist.
\end{proof}

\noindent
To get a transfer map, we use  an autonomous increasing homotopy between $H_1$ and $H_2$ and we deform it
into an increasing homotopy between $H_{1\delta} $ 
and  $H_{2\delta}$;  this induces a map
$$
\Z[u]\otimes_{\Z}SC_*(H_{1\delta})\rightarrow \Z[u]\otimes_{\Z}SC_*(H_{2\delta}).
$$
This map decreases the action (which is defined on the second factor) and commutes with the differential so it
 induces a map on the quotient 
 $$
 H\big((\Z[u]\otimes_{\Z}SC_*(H_{1\delta},\partial\bigr))\rightarrow H\bigl( \Z[u]\otimes_{\Z}C^{\rom{2}}(H_{2\delta}),\partial\bigr).
 $$
 This commutes with continuation maps.
 \begin{proposition}
For $\delta$ small enough, a continuation map in the homology defined from an $H_{2\delta}$ induces a continuation 
continuation map in the homology defined from $\iota_VH_{2\delta}$.
\end{proposition}
\begin{proof}
	One checks again that there is no parametrized Floer trajectory, corresponding to a homotopy, going from an orbit in $C^{\rom{1},\rom{2}}(H_{2\delta})$
	to an orbit in $C^{\rom{1},\rom{2}}(H'_{2\delta})$ with points in $\widehat{W}\setminus (U\cup V)$. 
	This is done as in the former proposition, using the fact that the existence of such a trajectory for all $\delta$'s
	would imply the existence of  such a broken trajectory for the autonomous Hamiltonian and we have proven in Proposition \ref{prop:transfer2} that this can not exist.
	\phantom{coucou}
\end{proof}

\noindent We thus get a transfer morphism
$$
\phi^{S^1}_{W,V} : SH^{S^1}(W,\lambda_W)\rightarrow SH^{S^1}(V,\lambda_V).
$$
and, on the quotient, the morphism
$$
\phi^{S^1,+}=\phi^{S^1,+}_{W,V}: SH^{S^1,+}(W,\lambda_W)\rightarrow SH^{S^1,+}(V,\lambda_V).
$$
By the same arguments as before, those morphisms compose nicely.
\begin{theorem}[Composition]\label{thm:fonctS1}
	Let $(V_1,\lambda_{V_1}) \subseteq (V_2,\lambda_{V_2})\subseteq (V_3,\lambda_{V_3})$ be Liouville domains with Liouville embeddings.
	Then the following diagram commutes:
	\begin{equation}\label{diagcommutatifSHS1+}
		\xymatrix{
			SH^{S^1,+}(V_3,\lambda_{V_3}) \ar[r]^{\phi^{S^1,+}_{V_3,V_2}} \ar@/_2pc/[rr]^{\phi^{S^1,+}_{V_3,V_1}} & SH^{S^1,+}(V_2,\lambda_{V_2}) \ar[r]^{\phi^{S^1,+}_{V_2,V_1}} & SH^{S^1,+}(V_1,\lambda_{V_1}).
		}
	\end{equation}
\end{theorem}

\subsection{Invariance of symplectic homology}

In this section, we study the invariance of the ($S^1$-equivariant) positive symplectic homology with respect to the choice of the Liouville vector field in a neighbourhood of the boundary. This  has been studied by Viterbo \cite{V}, Cieliebak \cite{C} and Seidel \cite{Seidel} in the case of the  symplectic homology.
\begin{lemma}
	Let $(W,\w,X)$ be a compact symplectic manifold with contact type boundary and let $k$ be a positive real number.
	Then $
		SH^\dagger(W,\w,X) = SH^\dagger(W,k\w,X)$,
	where $\dagger$ denotes any of the variants that we have considered $\emptyset,+, S^1$ or $(S^1,+)$.
\end{lemma}
\begin{proof}
	The symplectic completions are $(\widehat{W},\widehat{\w})$ and $(\widehat{W},k\widehat{\w})$; the chain complexes for a pair $(H,J)$ on $(\widehat{W},\widehat{\w})$ and the pair
	$(kH,J)$ on $(\widehat{W},k\widehat{\w})$ are the same, since the $1$ periodic orbits are the same, and the Floer trajectories satisfy the same equations; indeed $X_{H}^{\w} = X_{kH}^{k\w}$. Similarly, continuation maps are equivalent taking as homotopies $H_s$ and $kH_s$.
The result follows, observing that  $kH$ form a cofinal family.
\end{proof}

\noindent
For positive or $S^1$-equivariant positive homology, we  assume that $(W,\w,X)$ is a Liouville domain.
\begin{lemma}\label{lem:invsymplecto}
	Let $(W,\w, X)$ and $(W',\w',X')$ be two compact symplectic manifolds with contact type boundary. 	If there exists a symplectomorphism $\varphi:W\rightarrow W'$ such that $\varphi(\partial W) = \partial W'$, and such that $\varphi_{\star}(X)=X'$ on a neighbourhood of $\partial W$
	then $
		SH^\dagger(W,\w,X)\cong SH^\dagger(W',\w',X')$.
\end{lemma}
\begin{proof}
	We can extend $\varphi$ to a symplectomorphism $\widehat{\varphi}:\widehat{W}\rightarrow\widehat{W'}$ of the completions.
	 For $J'$ an almost complex structure on $\widehat{W'}$, we  take the corresponding almost complex structure $J$ on $\widehat{W}$ defined by $
		J_x := {\widehat{\varphi}^{-1}_{\star_x}} \circ J'_{\widehat{\varphi}(x)}\circ {\widehat{\varphi}_{\star_x}}$
and if $H'$ is a Hamiltonian on $\widehat{W'}$, we take  the Hamiltonian $H$ on $\widehat{W}$ defined by $H:=\widehat{\varphi}^{\star}H'$.
	Then the $1$ periodic orbirs are in bijection and so are the Floer trajectories.  The subfamily $\{\widehat{\varphi}^{\star}H'\}$ of Hamiltonians is cofinal 
	hence the conclusion.
\end{proof}
\begin{lemma}\label{lem:de^tlambda0}
	Let $(W,\lambda)$ be a Liouville domain.	Then for all $R\in\R^+$, we have
	\[
		SH^\dagger(W,\lambda ) \cong SH^\dagger\bigl(W\cup (\partial W\times[0,R]),\lambda'\bigr)
	\]
	where the $1$-form $\lambda'$ on $\partial W\times[0,R]$ is the restriction of the $1$-form $\widehat{\lambda}$, thus $(e^{\rho}\alpha)$ with $\alpha:=\lambda_{\vert_{\partial W}}$.\end{lemma}
\begin{proof}
	Denote by $\varphi_t^X$ the flow of $X$;  since $ \mathcal{L}_X\lambda=\lambda$ we have ${\varphi_t^X}^{\star}\lambda = e^t\lambda$.
	This gives a symplectomorphism
	\[
		\varphi_R^X: (W,e^R\w)\rightarrow\bigl(W\cup(\partial W\times[0,R]),\w'\bigr)
	\]
mapping the boundary $\partial W$ to the boundary $\{R\}\times \partial W$ and such that ${\varphi_R^X}^*\lambda=e^R\lambda$.
	One concludes by the two lemmas above.\\
	Explicitly, the diffeomorphism $\varphi_R^X: \widehat{W}\rightarrow \widehat{W}$ maps Hamiltonian vector fields as follows : $(\varphi_R^X)_*(X_{H'})=X_H$ when $H'=e^{-R}(\varphi_R^X)^*H$; hence $\varphi_R^X$  gives a bijection between $1$-periodic orbits of $X_{H'}$ and $1$-periodic orbits of $X_{H}$, and,  with suitable choices of $J$'s, a bijection between Floer trajectories between  $1$-periodic orbits of $X_{H'}$ and Floer trajectories between $1$-periodic orbits of $X_{H}$. Hence it yields an isomorphism
	\[
	SH^\dagger(W,e^{-R}(\varphi_R^X)^*H)\cong SH^\dagger\bigl(W\cup (\partial W\times[0,R]), H \bigr).
	\]
	 Furthermore, the diffeomorphism $\varphi_R^X$ intertwines  a continuation morphism defined by a homotopy  $H'_s$ to the corresponding continuation morphism defined by $H_s$ when again $H'_s=e^{-R}(\varphi_R^X)^*H_s$. This yields the isomorphism
mentioned above.
	\end{proof}
\begin{lemma}\label{lem:identitytransfer}
The transfer morphism  $$
		 SH^\dagger\bigl(W\cup (\partial W\times[0,R]),\lambda'\bigr)\mapsto SH^\dagger(W,\lambda) $$
is an isomorphism; it coincides with the natural identification of Lemma \ref{lem:de^tlambda0}.
\end{lemma}
\begin{proof}
	Let $H$ be an admissible Hamiltonian for $W\cup(\partial W\times[0,R])$.
	Consider the homotopy $H^1_s:=e^{-f(s)} {\varphi_{f(s)}^X}^{\star}H$ with $f:\R\rightarrow[0,R]$ a smooth function so that $H^1_{s}=H$ for large negative s and $H^1_{s}=\widetilde{H}:=e^{-R}(\varphi_R^X)^*H$ for large positive $s$.	The set of $1$ periodic orbits for $H^1_s$ is constant (since, as in the Lemma above, the diffeomorphism $\varphi_{f(s)}^X$ of the completion is a bijection between  $1$-periodic orbits of $X_{H^1_s}$ and
$1$-periodic orbits of $X_{H}$).
	This homotopy defines the ``transfer morphism''
	\[
		\phi : SH\bigl(W\cup(\partial W\times[0,R]),H\bigr)\rightarrow SH(W,\widetilde{H}).
	\]
	Let $\{H^{\eta}_s\}_{\eta\in[0,1]}$ be a family of homotopies (with non fixed endpoint) such that $H^0_s$ is the constant homotopy $H^0_s=H$  for  all $s$,  and such that all $H^{\eta}_s$ are of the form
	$e^{-f'(s,\eta)} {\varphi_{f'(s,\eta)}^X}^{\star}H$ with $f'(.,\eta) :\R\rightarrow[0,\eta R]$ and $f'(.,1)=f$. We have  $H^\eta_{+\infty}=e^{-\eta R} {\varphi_{\eta R}^X}^{\star}H=H^1_{f^{-1}(\eta R)}$.
	The set of $1$-periodic orbits of $H^{\eta}_s$ is in bijection with the set of orbits of $H$.
	We consider, for a given $\eta$, the space of Floer trajectories
         \[
		\Mod(H^{\eta}_s,J^{\eta}_s) := \bigcup_{\substack{(\gamma^{\eta}_-,\gamma^{\eta}_+)\in\Per(H^{\eta}_{-\infty})\times\Per(H^{\eta}_{+\infty})\\ \mu_{CZ}(\gamma^{\eta}_-)=\mu_{CZ}(\gamma^{\eta}_+)}}
		\Mod(\gamma^{\eta}_-,\gamma^{\eta}_+,H^{\eta}_s,J^{\eta}_s)
	\]
	and the parametrized moduli space $
		\Mod(\{H^{\eta}_s ,  J^{\eta}_s\}) := \bigcup_{\eta\in[0,1]}\Mod(H^{\eta}_s,J^{\eta}_s)$, 
	which could have boundaries for some $\eta\neq 0,1$.
	It defines  a cobordism between $\Mod(H^{0}_s,J^{0}_s)$ and $\Mod(H^{1}_s,J^{1}_s)$.
	Now $\Mod(H^{0}_s,J^{0}_s)=\Mod(H,J)$  is the space of constant trajectories $\{ \, u(s,\cdot) =\gamma_0 (\cdot)\,|\,\gamma_0\in\Per(H)\}$.
	Thus for small $\eta$'s, say $\eta\le \eta_0$,  the cobordism is a bijection, $\Mod(H^{\eta}_s,J^{\eta}_s)$ consists
	of  exactly one Floer  trajectory starting from each orbit in $\Per(H)$ and arriving at the corresponding
	orbit in $ \Per(H^\eta_{+\infty})$.
	The morphism induced by $H^{\eta_0}_s$ is thus the natural identification of periodic orbits. Hence the transfer 
	\[
		\phi : SH\bigl(W\cup(\partial W\times[0,R]),H\bigr)\rightarrow SH\bigl(W\cup(\partial W\times[0,R-\epsilon]),e^{\epsilon} {\varphi_{\epsilon}^X}^{\star}H\bigr)
	\]
	is the natural identification for $\epsilon= \eta_0 R$.
	Now we use the flow of the Liouville vector field, $\varphi_{\epsilon}^X$,  to carry all this construction further and we get the natural identification as the transfer morphism
	{\small{\[
		\phi : SH\bigl(W\cup(\partial W\times[0,R-\epsilon]),e^{\epsilon} {\varphi^{\epsilon}_X}^{\star}H\bigr)\rightarrow SH\bigl(W\cup(\partial W\times[0,R-2\epsilon]),e^{2\epsilon} {\varphi^{2\epsilon}_X}^{\star}H\bigr).
	\]}}
	By induction and functoriality, we get the result.
\end{proof}
\begin{lemma}\label{lem:isotopyfixedonboundary}
	Let $W$ be a compact symplectic manifold with contact type boundary.
	Let $\lambda_t,\, t\in[0,1]$ be an isotopy of Liouville forms on $W$ such that in a neighbourhood $U$ of the boundary, $\lambda_t=\lambda_0$.
	Then $
		SH^\dagger(W,\lambda_0) \cong SH^\dagger(W,\lambda_1)$.
\end{lemma}
\begin{proof}
	Remark that we do not require the $d\lambda_t$ to be equal.
	We define the time dependent vector field $X_t$ by $
		\iota(X_s)(d\lambda_s)=-\bigl(\tfrac{d}{dt}\lambda(t)_{\vert_s}\bigr)$
	and we denote by  $\varphi_t$ its flow.
	In the neighbourhood  $U$, the vector field vanishes, $X_s=0$,  and so $\varphi_1^{\star}\lambda_1=\lambda_1=\lambda_0$ on $U$ .
	Furthermore $\varphi_1^{\star}d\lambda_1=d\lambda_0$ because
	$
		\frac{d}{dt}\varphi_t^{\star}\lambda_t\Big|_s =\varphi_s^{\star}\Bigl(\frac{d\lambda_t}{dt}{\Big|_s}\Bigr) + \varphi_s^{\star}\Lie_{X_s}\lambda_s=d\Bigl(\varphi_s^{\star}\bigl(\lambda_s(X_s)\bigr)\Bigr).
	$
	This implies that the completions for $\lambda_0$ and $\varphi_1^{\star}\lambda_1$ are the same, therefore, by lemma \ref{lem:invsymplecto}, $
		SH^\dagger(W,\lambda_1)=SH^\dagger(W,\varphi^{\star}_1\lambda_1) = SH^\dagger(W,\lambda_0).$
		\end{proof}

\begin{theorem}\label{prop:invfunct}
	Let $W$ be a compact symplectic manifold with contact type boundary.
	Let $\lambda_t,\, t\in[0,1]$ be a homotopy of Liouville forms on $W$.
	Then
	\[
		SH^\dagger(W,\lambda_0) \cong SH^\dagger(W,\lambda_1).
	\]
\end{theorem}
To prove this Proposition, we use the following Proposition from Cieliebak and Eliashberg:
\begin{proposition}[\cite{CE}, Proposition 11.8]\label{prop:CE}
	Let $W$ be a compact symplectic manifold with contact type boundary.
	Let $\lambda_t,\, t\in[0,1]$ be a homotopy of Liouville forms on $W$.
	Then there exists a diffeomorphism of the completions $f:\widehat{W_0}\rightarrow\widehat{W_1}$ such that $f^{\star}\widehat{\lambda_1}-\widehat{\lambda_0}=dg$ where $g$ is a compactly supported function.
\end{proposition}
\begin{proof}[ Proof of Theorem \ref{prop:invfunct}]
	There exists a  real $\rho_0>0$ such that \\
	$supp(g)\subset W\cup(\partial W\times[0,\rho_0])$.
	We choose positive real numbers $\rho_1,\rho_0'$ and $\rho_1'$ such that $f^{-1}(W\cup(\partial W\times[0,\rho_1])$ contains $W\cup(\partial W\times[0,\rho_0])$,\\
	$~\qquad f^{-1}(W\cup(\partial W\times[0,\rho_1]))\subset W\cup(\partial W\times[0,\rho_0'])$ and \\
	$~\qquad W\cup(\partial W\times[0,\rho_0'])\subset f^{-1}(W\cup(\partial W\times[0,\rho_1']))$.
	The situation is represented in Figure \ref{fig:thesituation}.
	\begin{figure}
		\def\svgwidth{2\textwidth}
		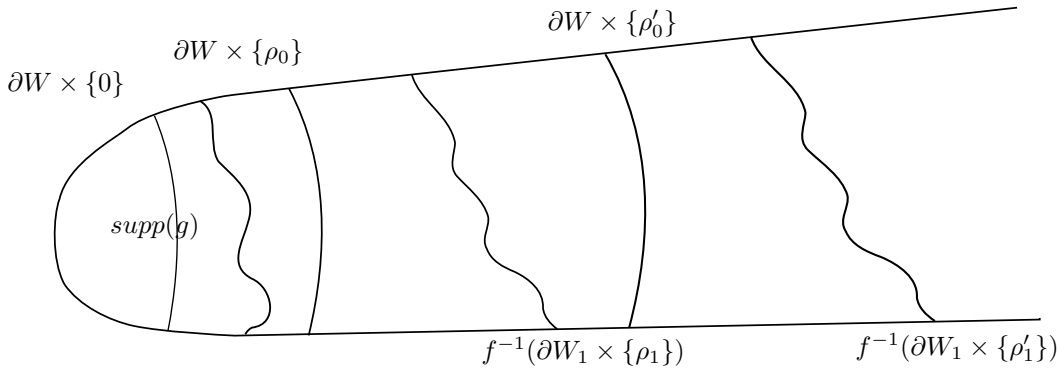
		\caption{The choice of $\rho_0$, $\rho_1$, $\rho_0'$ and $\rho'_1$}\label{fig:thesituation}
	\end{figure}\\	
	The diffeomorphism $f$ and the flow of $X_1$ on $\widehat{W}_1$ give
	\[
		\bigl(f^{-1}(W\cup(\partial W\times[0,\rho_1])), f^{\star}\widehat{\lambda_1}\bigr)\cong \bigl(W\cup(\partial W\times[0,\rho_1], \widehat{\lambda_1}\bigr) \cong (W,e^{\rho_1}\lambda_1).
	\]
	The completion of $\bigl(f^{-1}(W\cup(\partial W\times[0,\rho_1])), f^{\star}\widehat{\lambda_1}\bigr)$ coincides with $(\widehat{W}_0,\widehat{\lambda_0})$  since close to the boundary $f_{\star}X_0=X_1$.
We have 
	\begin{align*}
		SH(W,\lambda_1)&\cong SH\bigl(W\cup(\partial W\times[0,\rho_1]),\widehat{\lambda_1}\bigr)\quad\textrm{by Lemma \ref{lem:de^tlambda0}}\\
		&\cong SH\bigl(f^{-1}(W\cup(\partial W\times[0,\rho_1])), f^{\star}\widehat{\lambda_1}\bigr)\quad\textrm{by Lemma \ref{lem:invsymplecto}}\\
		&\cong SH\bigl(f^{-1}(W\cup(\partial W\times[0,\rho_1])), \widehat{\lambda_0} + dg\bigr)\quad\textrm{by Proposition \ref{prop:CE}}\\
		&\cong SH\bigl(f^{-1}(W\cup(\partial W\times[0,\rho_1]))=:W_1, \widehat{\lambda_0}\bigr)\quad\textrm{by Lemma \ref{lem:isotopyfixedonboundary}}.
	\end{align*}
	Denoting by $\varphi^{X_0}_t$ the flow of $X_0$ and by $W_0$ the manifold $W\cup(\partial W\times[0,\rho_0])$, we have
$\varphi^{X_0}_{\rho_1'-\rho_1}(W_1) = f^{-1}(W\cup(\partial W\times[0,\rho_1']))$
	and $\varphi^{X_0}_{\rho_0'-\rho_0}(W_0) = W\cup(\partial W\times[0,\rho_0'])$.
Using the functoriality of the transfer morphism, we get
	\[ \scriptsize{
		\xymatrix{
			SH(\varphi^{X_0}_{\rho_1'-\rho_1}(W_1),\widehat{\lambda_0}) \ar[r] \ar@/_1pc/[rr]_\cong & SH(\varphi^{X_0}_{\rho_0'-\rho_0}(W_0),\widehat{\lambda_0}) \ar@/_-1pc/[rr]^\cong \ar[r]
			& SH(W_1,\widehat{\lambda_0}) \ar[r]& SH(W_0,\widehat{\lambda_0});
		}}
	\]
	therefore $
		SH(W,\lambda_1) \,\cong \, SH(W_1,\widehat{\lambda_0})\,\cong\, SH(W_0, \widehat{\lambda_0})\,\cong\, SH(W,\lambda_0)$.
\end{proof}

Seidel in \cite{Seidel} has extended the definition of symplectic homology (and all its variants) to Liouville manifolds.
\begin{definition}[see for instance \cite{CE}]\label{def:liouvillemanifold}
	A \emph{Liouville manifold} is an exact symplectic manifold $(W,\w,X)$, where the vector field $X$ is an expanding Liouville vector field, i.e $\Lie_X\w=\w$ and $\varphi^X_t\w=e^t\w$ such that the vector field $X$ is complete and
 the manifold is convex in the sense that there exists an exhaustion $W=\cup_{k=1}^{\infty}W^k$ by compact domains $W_k\subset W$ with smooth boundaries along which $X$ is outward pointing.
\end{definition}
In the following we will denote a Liouville manifold either by $(W,\w,X)$ or by $(W,\lambda := \iota(X)\w)$.
The set $
	\operatorname{Skel}(V,\w,X) := \bigcup_{k=1}^{\infty}\bigcap_{t>0}\varphi^X_{-t}(W^k)$ 
is called the \emph{skeleton} of the Liouville manifold $(W,\w,X)$.
It is independent of the choice of the exhausting sequence of compact sets $W^k$.
A Liouville manifold $(W,\w,X)$ is said to be of \emph{finite type} if its skeleton is compact.
Every finite type Liouville manifold is the completion of a Liouville domain\footnote{We refer to the book by Cieliebak and Eliashberg for more details, \cite[Chapter 11]{CE}}.
\begin{definition}[\cite{Seidel}]
	Let $(W,\w,X)$ be a Liouville manifold non necessarily of finite type and let $W^k$ be an exhaustion by compact domains $W_k\subset W$ with smooth boundaries along which $X$ is outward pointing
	such that $W^k\subset W^{k+1}$.
	The symplectic homology (and its variants) of $(W,\lambda)$ is defined as the inverse limit of the symplectic homologies of $(W^k,\lambda_{|_{W^k}})$
	\[
		SH^{\dagger}(W,\lambda) := \lim_{\longleftarrow}SH^\dagger(W^k,\lambda_{|_{W^k}}).
	\]
	The morphisms appearing in this inverse limit are the transfer morphisms.
This definition is independent of the chosen exhaustion.
Remark that in the case of finite type Liouville manifolds, this definition coincides with the previous one.
\end{definition}
\begin{proposition}
	Let $(W_0,\lambda_0)$ and $(W_1,\lambda_1)$ be two Liouville manifolds not necessarily of finite type.
	Assume there exists an exact symplectomorphism $f:W_0\rightarrow W_1$ i.e. such that $f^{\star}\lambda_1-\lambda_0=dg$ with $g$ a function on $W_0$.
	Then $
		SH^\dagger(W_0,\lambda_0)\cong SH^\dagger(W_1,\lambda_1).$
\end{proposition}
\begin{proof}
	Let $W_0^k$ be an exhaustion for $W_0$ and $W_1^k$ be an exhaustion for $W_1$ such that for all $k$,
	\[
		W_0^k\subset f^{-1}(W_1^k) \subset W_0^{k+1}
	\]
	where the inclusion at each level means the inclusion in the interior of the next compact space.
	Let $\eta$ be a smooth function $\eta:W_0\rightarrow[0,1]$ such that $\eta=1$ in a neighbourhood of $\cup_{k=1}^{\infty}f^{-1}(\partial W_1^k)$ and
	$\eta=0$ in a neighbourhood of $\cup_{k=1}^{\infty}\partial W_0^k$.
	We define the $1$-form $\lambda:= \lambda_0 + d(\eta g)$ on $W_0$.
		We have
	\[
		SH(W_0^k,\lambda_0)\cong SH(W_0^k,\lambda)\quad\textrm{and}\quad SH(W_1^k,\lambda_1)\cong SH\bigl(f^{-1}(W_1^k),\lambda\bigr).
	\]
	The functoriality of the transfer morphism implies that the following diagram is commutative:
	
	\[
		\tiny{\xymatrix{
			\cdots \ar[r] \ar@/_-1pc/[rr] & SH\bigl(f^{-1}(W_1^{k+1}),\lambda\bigr) \ar[r] \ar@/_1pc/[rr] & SH(W_0^{k+1},\lambda) \ar@/_-1pc/[rr] \ar[r]
			& SH\bigl(f^{-1}(W_1^k),\lambda\bigr) \ar[r]  \ar@/_1pc/[rr]& SH(W_0^k,\lambda)\ar[r] & \cdots
		}}.
	\]
	
	Therefore, $ SH(W_0,\lambda_0) \cong {\displaystyle{\lim_{\longleftarrow}}}SH(W_0^k,\lambda_0) \cong {\displaystyle{\lim_{\longleftarrow}}}SH(W_0^k,\lambda)$\\
	$
		 \cong {\displaystyle{\lim_{\longleftarrow}}}SH\bigl(f^{-1}(W_1^k),\lambda\bigr)  \cong{\displaystyle{\lim_{\longleftarrow}}}SH(W_1^k, f_{\star}\lambda)
		 \cong {\displaystyle{\lim_{\longleftarrow}}}SH(W_1^k, \lambda_1)  \cong SH(W_1, \lambda_1).$
\end{proof}

\noindent The above result may be extended thanks to the following Lemma:
\begin{lemma}[\cite{BEE}, see also \cite{CE}, Lemma 11.2]
	Any symplectomorphism between finite type Liouville manifolds $f:(W_0,\lambda_0)\rightarrow(W_1,\lambda_1)$ is diffeotopic to an exact symplectomorphism.
\end{lemma}
We have thus proven the invariance Theorem \ref{thm:invLiouville} stated in the introduction.

\subsection{Invariance of the homology of contact fillings}
In this section we shall prove  Theorem \ref{prop:invtcontact}  giving an invariant of the contact structure.
\begin{lemma}[\cite{C}]\label{lem:cieliebak}
	Let $(\alpha_t)_{t\in[0,1]}$ be a smooth family of contact forms on a closed manifold $M$ of dimension $2n-1$.
	Then there exists $R>0$ and a non-decreasing function $f:[0,R]\rightarrow[0,1]$ such that $f\equiv0$ close to $\rho=0$ and $f\equiv1$ close to $\rho=R$
	and $d\bigl(e^{\rho}\alpha_{f(\rho)}\bigr)$ is symplectic on  $M\times[0,R]$.
\end{lemma}
\begin{proof}
	The proof is a computation: \\$
		d\bigl(e^{\rho}\alpha_{f(\rho)}\bigr) = e^{\rho}d\rho\wedge\alpha_{f(\rho)} + e^{\rho}d\alpha_{f(\rho)} + e^{\rho}f'(\rho)d\rho\wedge\dot{\alpha}_{f(\rho)}
$ and \\ $\Bigl(d\bigl(e^{\rho}\alpha_{f(\rho)}\bigr)\Bigr)^n = ne^{n\rho}\Bigl(d\rho\wedge\bigl(\alpha_{f(\rho)}+f'(\rho)\dot{\alpha}_{f(\rho)}\bigr)\wedge\bigl(d\alpha_{f(\rho)}\bigr)^{n-1}\Bigr);
$ thus \\$d\bigl(e^{\rho}\alpha_{f(\rho)}\bigr)$ is symplectic if and only if $\bigl(\alpha_{f(\rho)}+f'(\rho)\dot{\alpha}_{f(\rho)}\bigr)\bigl(R_{\alpha_{f(\rho)}}\bigr) >0$.
	This is true if $f'$ is small.
\end{proof}
\begin{lemma}
	If $(M,\xi)$ is a compact contact manifold which is exactly fillable by a Liouville domain $(W,\lambda_0)$ (i.e. $\partial W=M$ and $\xi=\ker\alpha_0$ where $\alpha_0=\lambda_0\big|_M$)
	then, for any contact form $\alpha_1$ such that $\xi=\ker\alpha_1$ (and $\alpha_1$ defines the same orientation on $M$), there exists a homotopy of Liouville form $\lambda_s, s\in[0,1]$ on $W$ such that $\lambda_1\big|_M=\alpha_1$.
\end{lemma}
\begin{proof}
	Since $\alpha_1=e^g\alpha_0$, for a smooth function $g$ on $M$, we consider the smooth family of contact forms $\alpha_t=e^{tg}\alpha_0, t\in[0,1]$.\\
	We define on $W\cup M\times[0,R]\subset\widehat{W}$ the $1$-form $\tilde{\lambda}$ such that $
		\tilde{\lambda}= \lambda_0$  on $M$ and  $ \tilde{\lambda}=e^{\rho}\alpha_{f(\rho)}$  on $M\times[0,R]$
	with $f$ as in Lemma \ref{lem:cieliebak}, so that $d\tilde{\lambda}$ is symplectic.
	The flow $\varphi^{X_0}_{-r}$ of the vector field $X_0$, where $\iota(X_0)d\lambda_0=\lambda_0$,   induces a diffeomorphism from $W\cup M\times[0,r]$ to $W$.
	The pull-back by this flow of $e^{-r}\tilde{\lambda}$ gives the desired $\lambda_{f(r)}$.
\end{proof}

\noindent
Combining with Theorem \ref{prop:invfunct} and Theorem \ref{thm:computing}, this yields 
\begin{lemma}\label{lem:invctc}
	Let $(M_0,\xi_0)$ be a contact manifold that is exactly fillable by the Liouville domains $(W_0,\lambda_0)$.
	Assume that there exists a (oriented) contact form $\tilde{\alpha}_0$ on $M_0$ such that all periodic Reeb orbits are nondegenerate and their Conley-Zehnder index have all the same parity.
	Then
	\[
		SH^{S^1,+}(W_0,\lambda_0) = \bigoplus_{\gamma\in\Per(R_{\tilde{\alpha}_0})} \Q\langle\gamma\rangle
	\]
	where $\Per(R_{\tilde{\alpha}_0})$ denotes the set of periodic Reeb orbits on $(M_0,\tilde{\alpha}_0)$.{{\hfill$\square$\\ }}
\end{lemma}
\begin{proof}[ Proof of Theorem \ref{prop:invtcontact}]
	Given the contactomorphism $\varphi:(M_0,\xi_0)\rightarrow (M_1,\xi_1)$ and the contact form $\tilde{\alpha}_0$, we define the form $\tilde{\alpha}_1:= (\varphi^{-1})^{\star}\tilde{\alpha}_0$;
	it is a contact form on $M_1$ and its periodic orbits are non degenerate, in bijection with those of $\tilde{\alpha}_0$.
	{The isomorphism preserves the Conley-Zehnder index if the orbit is null-homologous in the boundary but, in general, the isomorphism sends the framing (trivialisation) chosen for the orbit of $R_{\tilde{\alpha}_0}$ on a framing for the orbit of $R_{\tilde{\alpha}_1}$.
	There is a $\Z$-action on the framings (corresponding to the number of twists).
	For any integer number, the Conley-Zehnder index will change by an even number (see \cite{MacDuffSalamon, CZ, JGutt}).}
	We apply twice Lemma \ref{lem:invctc}; once for $(W_0,\lambda_0,\tilde{\alpha}_0)$ and for $(W_1,\lambda_1,\tilde{\alpha}_1)$.
	{We have therefore an isomorphism between the two graded modules with eventually shifts of degree of the generators given by the choice of framings.}
\end{proof}

\noindent
\subsection{Application : Non isomorphic contact structures on $S^{4m+1.}$}\label{section:Usti}
A first application of our results is to give a proof   of Ustilovsky's Theorem.
\begin{definition}[The Brieskorn spheres]\label{sect:Brieskorn}
	The Brieskorn manifold, denoted $\Sigma(a_0,\ldots,a_n)$, with all $a_i\geq2$ positive integers, is defined as the intersection of the unit sphere
	$S^{2n+1}\subset\C^{n+1}$ with the singular hypersurface
	$\{(z_0,\ldots,z_n)\in\C^{n+1}\,|\,z_0^{a_0}+\cdots+z_n^{a_n}=0\}$ in $\C^{n+1}$. \\ It is a smooth $2n-1$-dimensional manifold which
	 admits a contact form\\
	$
		\alpha=\frac{i}{8}\sum_{j=0}^n a_j(z_jd\overline{z_j} - \overline{z_j}dz_j)
	$
	with corresponding Reeb vector field\\
	$
		R_{\alpha} = \left(\frac{4i}{a_0}z_0,\ldots,\frac{4i}{a_n}z_n\right).
	$
For any odd number $n=2m+1$ and any $p\equiv\pm1\mod8$, the Brieskorn manifold $\Sigma(p,2,\ldots,2)$ is diffeomorphic to the standard sphere $S^{4m+1}$ \cite{Brieskorn}.
	One defines the contact structures $\xi_p$ on $S^{4m+1}$ defined as the kernel of the contact form $\alpha_p$ with
	\[
		\alpha_p := \frac{ip}{8} (z_0d\overline{z_0} - \overline{z_0}dz_0) + \frac{i}{4}\sum_{j=1}^n (z_jd\overline{z_j} - \overline{z_j}dz_j).
	\]
The fact that the Brieskorn spheres are exactly fillable can be found, for instance, in the book of Geiges \cite{Geiges}.
\end{definition}
\begin{proposition}\label{Brieskornfillings} For $p_1\neq p_2$, the positive $S^1$ equivariant homologies of 
symplectic fillings of the Brieskorn spheres are different.
\end{proposition}
 \begin{proof}
We consider the  description of the chain complex for those  homologies in terms
of good periodic orbits of the Reeb  vector field, graded by minus their Conley indices.
We shall show that all Conley-Zehnder indices are even.
To compute them, the first thing to do is to build an explicit perturbation of the contact form so that all periodic Reeb orbits are non degenerate. We proceed as in \cite{U}.
	We make the change of coordinates
	 \[
	 	w_0=z_0,\,w_1=z_1\,
		\left(\begin{array}{c}
			w_{2j}\\
			w_{2j+1}
		\end{array}\right)
		=\frac{1}{\sqrt{2}}
		\left(\begin{array}{cc}
			1&i\\
			1&-i
		\end{array}\right)
		\left(\begin{array}{c}
			z_{2j}\\
			z_{2j+1}
		\end{array}\right),
		\textrm{ for} j\geq1;
	\]
	then {{$
		\Sigma(p,2,\ldots2) = \left\{ w\in\C^{n+1}\,\bigg|\,w_0^p+w_1^2+2\sum_{j=1}^mw_{2j}w_{2j+1}=0,\,\abs{w}^2=1\right\}.$}}
	Consider the real positive function $f:\Sigma(p,2,\ldots2)\rightarrow \R$ given by
	$$
		f(w)=\abs{w}^2+\sum_{j=1}^m\epsilon_j\bigl(\abs{w_{2j}}^2-\abs{w_{2j+1}}^2\bigr),\quad\textrm{where } 0<\epsilon_j<1.
	$$
	The contact form $f\alpha$ defines the same contact structure on $\Sigma(p,2,\ldots2)$ as $\alpha$ and its associated Reeb vector field is given by
	{\small{\[
		R_{f\alpha}(w) = \Bigl(\tfrac{4i}{p}w_0,2iw_1,2i(1+\epsilon_1)w_2,2i(1-\epsilon_1)w_3,\ldots,2i(1+\epsilon_m)w_{n-1},2i(1-\epsilon_m)w_n\Bigr).
	\]}}
	
	If all the $\epsilon_j$ are irrational and linearly independent over $\Q$, the only periodic orbits are\\
$\gamma_0(t) = \left(re^{\frac{4it}{p}},ir^{\frac{p}{2}}e^{2it},0,\ldots,0\right)$, for $r>0, r^p+r^2=1,\, 0\leq t\leq p\pi$,\\
$\gamma_j^+(t) = \Bigl(0,\ldots,0,\underbrace{e^{2it(1+\epsilon_j)}}_{2j},0,\ldots,0\Bigr)$, for $ 0\leq t\leq \frac{\pi}{1+\epsilon_j},\, j=1,\ldots,m$;\\
$\gamma_j^-(t) = \Bigl(0,\ldots,0,\underbrace{e^{2it(1-\epsilon_j)}}_{2j+1},0,\ldots,0\Bigr)$, for $ 0\leq t\leq \frac{\pi}{1-\epsilon_j},\, j=1,\ldots,m$\\
	and all their iterates, $\gamma_0^N$,  ${\gamma_j^+}^N$, ${\gamma_j^-}^N$, for all $N\geq1$.
	Their Conley-Zehnder index is given by
	{\small{\begin{align*}
		\mu_{CZ}\bigl(\gamma_0^N\bigr) &= 2Np(n-2)+4N;\\
		\mu_{CZ}\Bigl({\gamma_j^{\pm}}^N\Bigr) &= 2\left\lfloor\frac{2N}{p(1\pm\epsilon_j)}\right\rfloor + 2\left\lfloor\frac{N}{1\pm\epsilon_j}\right\rfloor + 2\sum_{\substack{k=1\\k\neq j}}^{m}\left(\left\lfloor\frac{N(1+\epsilon_k)}{1\pm\epsilon_j}\right\rfloor + \left\lfloor\frac{N(1-\epsilon_k)}{1\pm\epsilon_j}\right\rfloor\right) +n-1.
	\end{align*}}}
All indices have the same parity, thus applying Theorem \ref{thm:computing}, the $S^1$-equivariant positive symplectic homologies
are generated by the periodic orbits of the Reeb vector field graded by their Conley indices.
If $p_1\neq p_2$, those positive $S^1$-equivariant symplectic homologies are different as proven in \cite{U}.
\end{proof}

\noindent A more complete description of Brieskorn manifolds and computations of the symplectic homologies  can be found in the paper by Kwon and van Koert \cite{KvK}.

\begin{cor}[Ustilovsky, \cite{U}]
	For each natural number $m$, there exist infinitely many pairwise non isomorphic contact structures on $S^{4m+1}$. \end{cor}
\begin{proof}
We see that one can build contact structures on $S^{4m+1}$, which  are exactly fillable, but which do not yield isomorphic  $SH^{S^1,+}$ homologies of the filling.
 The result then follows from Theorem \ref{prop:invtcontact}.
The contact structures in question are those defined by the Brieskorn spheres as above.
The fact that the homologies are different  follows from proposition \ref{Brieskornfillings}. 
\end{proof}


\section{On the minimal number of periodic Reeb orbits}\label{chap:orbits}

We now use the properties of positive $S^1$-equivariant symplectic homology to get results on the minimal
number of geometrically distinct periodic Reeb orbits on some contact manifolds.
\subsection{Minimal number of periodic Reeb orbits on a hypersurface in $\R^{2n}$}
  We use the transfer morphism to give an alternative  proof of a result by Ekeland and Lasry concerning the number of simple periodic Reeb orbits on a hypersurface in $\R^{2n}$,
pinched between two spheres, endowed with the restriction of the standard contact form on $\R^{2n}$.
\begin{theorem}[Ekeland, Lasry, \cite{EL,E}]\label{ekeland}
	Let $\Sigma$ be a contact type hypersurface in $\R^{2n}$.
	Let $\xi = \ker \alpha$ be the contact structure induced by the standard contact form on $\R^{2n}$.
	Assume there exists
	numbers $0<R_{1}\leq R_{2}$ such that:
	\begin{equation*}
		\forall x \in \Sigma , \quad R_{1}\leq \Vert x
			\Vert\leq R_{2} \qquad\textrm{with }
		\frac{R_{2}}{R_{1}} < \sqrt{2}
	\end{equation*}
	Assume also that $\forall x \in \Sigma, \quad T_{x}\Sigma \cap B_{R_{1}}(0) = \emptyset$.
	Assume moreover that all periodic Reeb orbits are non degenerate.
	Then $\Sigma$ carries at least $n$ geometrically distinct periodic Reeb orbits.
\end{theorem}
\begin{remark}
	The assumption $\forall x \in \Sigma,\, T_{x}\Sigma \cap B_{R_{1}}(0) = \emptyset$ (which is weaker than convexity) can be stated as
	\begin{equation}\label{hyp:rmk}
		\langle \nu_{\Sigma}(z),z\rangle > R_{1}, \quad \forall z \in \Sigma
	\end{equation}
	where $\nu_{\Sigma}(z)$ is the exterior normal vector of $\Sigma$ at point $z$ and $\langle\cdot,\cdot\rangle$ denotes the Euclidean scalar product on $\R^{2n}$.
\end{remark}
\begin{proof}
We consider ellipsoids, very close to the spheres, \\
$S'_{R_1}=\left\{ \sum_{i=1}^{n} a_{i}^{-1}\bigl((x^{i})^2+ (y^{i})^2\bigr)=R_1^2\,\right\}$ and \\$S'_{R_2}=\{ \sum_{i=1}^{n} a_{i}^{-1}\bigl((x^{i})^2+ (y^{i})^2\bigr)=R_2^2\,\}$, with $a_1<\cdots<a_n$  real numbers arbitrarily close to $1$ and rationally independent,  and we denote by 
 $\widetilde{S'_{R_1}}$, $\widetilde{\Sigma}$ and $\widetilde{S'_{R_2}}$ the compact regions in $\R^{2n}$ bounded respectively by $S'_{R_1}$, $\Sigma$ and $S'_{R_2}$, endowed with the restriction of the standard symplectic form $\w$ on $\R^{2n}$.
 We take the parameters $a_i$ sufficiently close to $1$ so that  we have the inclusion
 $$
 \widetilde{S'_{R_1}} \subset \widetilde{\Sigma} \subset \widetilde{S'_{R_2}}
 $$
of Liouville domains.
The contact form on the boundaries is the one induced by $\iota(X_{rad})\w $, where $X_{rad}$ is the radial vector field $X_{rad}=\thalf\sum x^i\partial_{x^i}+y^i\partial_{y^i}$.
The completion of those Liouville domain is $(\R^{2n},\w)$.
	By Theorem \ref{thm:fonctS1}, the transfer morphisms yields the following commutative diagram:
	\begin{equation}
		\xymatrix{
			SH^{S^1,+}(\widetilde{S'_{R_{2}}},\w) \ar[r]^{\phi} \ar@/_1.5pc/[rr]^\cong & SH^{S^1,+}(\widetilde{\Sigma}, \w) \ar[r] & SH^{S^1,+}(\widetilde{S'_{R_{1}}},\w).
		}
	\end{equation}
         We can consider the positive $S^1$-equivariant symplectic homology truncated by the action at level $\leq T$, $SH^{S^1,+,T}$.
         Since all Floer trajectories inducing the morphisms lower the action, we still have the commutative diagram for the truncated positive invariant symplectic homology:
	\begin{equation}\label{diagcommutatifHCT}
		\xymatrix{
			SH^{S^1,+,T}(\widetilde{S'_{R_{2}}},\w) \ar[r]^{\phi} \ar@/_1.5pc/[rr]^\cong & SH^{S^1,+,T}(\widetilde{\Sigma}, \w) \ar[r] & SH^{S^1,+,T}(\widetilde{S'_{R_{1}}},\w).
		}
	\end{equation}
	where we have chosen a number $T$ such that
	\begin{equation}
		\pi a_{n} R_{2}^{2} < T < 2\pi a_{1} R_{1}^{2}.
	\end{equation}
	This is possible thanks to  the ``pinching'' hypothesis $\frac{R_{2}}{R_{1}}<\sqrt{2}$.\\
	
	\noindent By Theorem \ref{thm:computing}, $SH^{S^1,+,T}(\widetilde{S'_{R_{2}}},\w)$ is generated by $n$ elements $$u^0\otimes\gamma_{\M}^1,\ldots u^0\otimes\gamma_{\M}^n$$ corresponding to
	$n$ simple periodic Reeb orbits on $S'_{R_2}$, $\gamma^1,\ldots,\gamma^n$, of action $\pi a_{1}R_{2}^{2},\ldots, \pi a_{n}R_{2}^{2}$.
	The analogous is true for $SH^{S^1,+,T}(\widetilde{S'_{R_{1}}},\w)$ with actions $\pi a_{1}R_{1}^{2},\ldots, \pi a_{n}R_{1}^{2}$.
	
	By \eqref{diagcommutatifHCT}, $SH^{S^1,+,T}(\widetilde{\Sigma}, \w)$ is thus of rank at least $n$.
	All applications in the above diagrams decrease the action thus the action of each of those $n$ generators in $SH^{S^1,+,T}(\widetilde{\Sigma}, \w) \cap \im(\phi)$ is pinched between
	$\pi a_{1} R_{1}^{2}$ and $\pi a_{n}R_{2}^{2} < 2\pi a_{1} R_{1}^{2}$.
	
	By Corollary \ref{cor:complexSHS1+}, the only generators that may appear in  $SH^{S^1,+,T}(\widetilde{\Sigma}, \w)$ are elements of the form $u^0\otimes\gamma_{\M}$ with $\gamma$ a good Reeb orbit on $\Sigma$.
	
	It remains  to prove that the $n$ elements in the image of $\phi$ are geometrically distinct.
	By the pinching condition on their action, we know that they are not iterate one from another but we still need to prove that two of them can not be the  iterates of a same  orbit of smaller action.
	This we do by proving that the smallest possible action for any periodic Reeb orbit on $\Sigma$ is greater than $\pi a_{1} R_{1}^{2}$.
	Let $\gamma : [0,T] \rightarrow \Sigma$ be a simple periodic Reeb orbit.
	We have :
	\begin{align}
		2T & = \int_0^T \alpha_{\gamma(t)}\bigl(\dot{\gamma}(t)\bigr)dt  
	            = \int_{0}^{T} \langle \dot{\gamma}(t), J\gamma(t)\rangle dt \qquad \textrm{since } \alpha_{x}(X_{x}) = \tfrac{1}{2} \langle X_{x}, Jx \rangle\nonumber\\
		& =  \int_{0}^{T} \langle \dot{\gamma}(t), J\bar{\gamma}(t)\rangle dt \qquad \textrm{with } \bar{\gamma}(t) := \gamma(t) -\frac{1}{T} \int_{0}^{T} \gamma(t) dt\nonumber\\
		 &\leq  \| \dot{\gamma}\|_{L^{2}}  \| \bar{\gamma}\|_{L^{2}} \nonumber\\
		& \leq   \| \dot{\gamma}\|_{L^{2}}^{2} \tfrac{T}{2\pi} \qquad \textrm{via the Wirtinger's inequality}\nonumber\\
		& =  \frac{T}{2\pi}\int_{0}^{T} \|\dot{\gamma}(t)\|^{2} dt
		=  \frac{T}{2\pi}\int_{0}^{T} \|(R_{\alpha})_{\gamma(t)}\|^{2} dt \label{borneaction}
	\end{align}
	For any point $x$ in $\Sigma$, the norm  of the Reeb vector field is bounded by $\|(R_{\alpha})_{x}\| \leq \frac{2}{R_{1}}$. 
	Indeed,  $R_{\alpha}$ is proportional to $J\nu_{\Sigma}$ since $\iota(J\nu_{\Sigma})d\alpha=0$ because $\iota(J\nu_{\Sigma})d\alpha(Y)=\w(J\nu_{\Sigma},Y)=-\langle \nu_{\Sigma},Y\rangle=0$ for all $Y\in T\Sigma$.
	Thus $R_{\alpha}=cJ\nu_{\Sigma}$ with $|c|=\|R_{\alpha}\|$.
	But $\alpha_x({R_{\alpha}}_x)=1=\thalf\langle c_xJ\nu_{\Sigma}(x),Jx\rangle=\frac{c_x}{2}\langle\nu_{\Sigma}(x),x\rangle$.
	Therefore, by assumption \eqref{hyp:rmk}, $c_x=\frac{2}{\langle\nu_{\Sigma}(x),x\rangle}\leq \frac{2}{R_{1}}$.
	And thus \eqref{borneaction} $\leq \frac{4}{R_{1}^{2}} T \frac{T}{2\pi}$.
	Then $2T \leq 2T \frac{T}{\pi R_{1}^{2}}$ and we reach the conclusion
	\[
		T \geq \pi R_{1}^{2}.
	\]
	Hence the conclusion of the Theorem.
\end{proof}

\noindent We have to assume here that all periodic Reeb orbits are non degenerate; this hypothesis is not needed in the original proof.
The original proof of Theorem \ref{ekeland} uses variational methods that work only in $\R^{2n}$.

\subsection{Reeb orbits on hypersurfaces in negative line bundles}
Our framework here is   a complex line bundle $\mathcal{L}\stackrel{\pi}{\rightarrow} B^{2n}$ over   a closed symplectic manifold $(B^{2n},\w_B)$,
endowed with a Hermitian structure $h$ and a connection $\nabla$.
We assume $\mathcal{L}$ to be  \emph{negative}  i.e. $
	c_1(\mathcal{L}) = -\kappa[\w_B]$
for a real number $\kappa>0$.
The  \emph{transgression 1-form}, $\theta^{\nabla} \in \Omega^{1}\bigl(\mathcal{L}\setminus O_{\mathcal{L}}, \R\bigr)$ is 
defined by
\begin{equation}
	\left\{\begin{array}{rcl}
		\theta^{\nabla}_u(u) & = & 0 ,\quad  \theta^{\nabla}_u(iu) = \frac{1}{2\pi} \qquad u\in \mathcal{L}\setminus O_{\mathcal{L}} \\
		\theta^{\nabla}\vert_{H^{\nabla}} & \equiv & 0 \ \ \textrm{where } H^{\nabla} \textrm{ is the horizontal distribution}.
	\end{array}\right.
\end{equation}
We have 
\[
	d\theta^{\nabla} = \kappa\pi^{\star}\w_B.
\]
We denote by $r:{\mathcal{L}} \rightarrow \R : u\mapsto h_{\pi(u)}(u,u)^{\half}=:\vert u\vert$ the radial function on the fiber.
Observe that $d(r^2\theta^{\nabla})$ is symplectic except on the zero section $O_{\mathcal{L}}$.
We want to have information about the minimal number of periodic orbits of the Reeb vector field on a hypersurface
 in $\mathcal{L}\setminus O_{\mathcal{L}}$ endowed with the contact form defined by the restriction of $(r^2\theta^{\nabla})$ to $\Sigma$.

\begin{proof}[ of Proposition \ref{propintro:orb}]
We start by determining Reeb orbits on the circle bundle with varying radius.
  Let $f : B \rightarrow \R$ be  a smooth function.
Define the contact hypersurface
$$
\Bigl(S_{e^f}= \{ u\in \mathcal{L} \, \vert \, \vert u \vert = e^{ f(\pi(u))} \},\alpha:=(r^2\theta^{\nabla})_{\vert_{S_{e^f}}}\Bigr).$$ 
The Reeb vector field on $S_{e^f}$ is given by:
\begin{equation}
	R_{{\alpha}} =e^{- 2f(\pi(u))} \Bigl(2\pi \partial_\theta+ \frac{2}{\kappa} \bar{X}_{f}\Bigr)
\end{equation}
where $ \partial_\theta$ is the infinitesimal rotation in the fiber ($ \partial_\theta$ at the point $u$ identifies with ${iu}$),  where $X_f$ is the Hamiltonian vector field on $B$ corresponding to the function $f$ (i.e.
$\iota({X}_{f})\w_B = df $) and where $\bar{X}$ denotes the horizontal lift of a vector $X\in TB$.
Periodic Reeb  orbits  correspond to the critical points of $f$.
Thus, using Morse's inequalities, we have:
	if $\Sigma$ is a contact type hypersurface in $\mathcal{L}$ such that the intersection of $\Sigma$ with each fiber is a circle,
	and if the contact form is the restriction of $r^2\theta^{\nabla}$,
	then $\Sigma$ carries at least $\sum_{i=0}^{2n} \beta_{i}$ geometrically distinct periodic Reeb orbits, where  $\beta_{i}$ denote the Betti numbers of $B$.
\end{proof}
	
We are now ready  to prove theorem \ref{thmintro:orb} for 
 a contact type hypersurface $\Sigma$ in a negative line bundle $\mathcal{L}$ over a symplectic manifold $B$, 
 when it is ``pinched'' between two circle bundles $S_{R_1}$ and $S_{R_2}$ of radii $R_1$ and $R_2$ such that $0<R_1<R_2$ and $\frac{R_2}{R_1}<\sqrt{2}$,
 and when there exists a Liouville domain $W'$ (such that its first Chern class vanishes on all tori)
whose boundary coincides with the circle bundle $S_{R_1}$. We endow $\Sigma$ with the contact form $\alpha$ induced by $r^2\theta^\nabla$ on  $\mathcal{L}$.
We assume that  the minimal action of any periodic Reeb orbit on $\Sigma$ is  bounded below by $R_1^2$.
We also assume that there exists a Morse function $f:B\rightarrow\R$ such that the set of indices of all critical points of $f$ is lacunary.
We want to prove that $\Sigma$ carries at least $\sum_{i=0}^{2n} \beta_{i}$ geometrically distinct periodic Reeb orbits.\\
\begin{proof}[ of theorem \ref{thmintro:orb}]
The proof is the same as for Theorem \ref{ekeland} using transfer morphisms for Liouville domains. 
We see the hypersurfaces as lying in the completion of the Liouville domain $W'$ which we assumed to exist.
We find a small $\epsilon$ so that the convex domain ${\widetilde{\Sigma}}$ bounded  by the hypersurface  $\Sigma$ 
is such that
$$
{\widetilde{S}}_{R_1e^{\epsilon f}}\subset {\widetilde{\Sigma}}  \subset {\widetilde{S}}_{R_2e^{\epsilon f}}\subset \widehat{W'}
$$ 
where ${\widetilde{S_{f}}}$ is the domain bounded by $S_f$. We can compute the positive $S^1$-equivariant symplectic homology, which is spanned by  periodic orbits of the Reeb vector field by Theorem \ref{thm:computing}.
This is possible by the pinching condition. One uses then the transfer morphisms with truncated action. We have seen that there are $\sum_{i=0}^{2n} \beta_{i}$ simple  periodic orbits on ${{S}}_{R_1e^{\epsilon f}}$ whose actions are  very close to  $R_1^2$ and 
the same number of  simple  periodic orbits on ${{S}}_{R_2e^{\epsilon f}}$ whose actions are  very close to  $R_2^2$. The transfer morphism imply the existence of at least $\sum_{i=0}^{2n} \beta_{i}$  periodic orbits on 
$\Sigma$ with action between $R_1^2$ et $R_2^2$. Since we have assumed here that the minimal action of any periodic Reeb orbit on $\Sigma$ is bounded below by $R_1^2$, those orbits are geometrically distinct.
\end{proof}

\noindent In this Theorem, the assumption on the existence of a Morse function all of whose critical points have Morse indices of the same parity is of a technical nature.
Its purpose is to bring the situation within the scope of Theorem \ref{thm:computing}, which is our tool for computing the positive $S^1$-equivariant symplectic homology.
The lower bound on the period of any periodic Reeb orbit is semi-technical; 
it is now the only way we have to distinguish the images of the orbits.
The ``pinching'' assumption is more conceptual, its main implication is that the ``$n$ first generators'' of the positive $S^1$-equivariant symplectic homology are simple orbits.
\begin{exemple}[Tautological complex line bundle over $\C P^{n-1}$]~\\
We consider the tautological complex (negative) line bundle over $\C P^{n-1}$,
$\mathcal{O}(-1)\longrightarrow \C P^{n-1}$.  The corresponding disk bundle,
which is the canonical disk bundle over $\C P^{n-1}$, is canonically isomorphic to
the ball blown up at the origin,
$\hat{B}^{2n} := \Bigl\{ \bigl(z,[t]\bigr) \in \C^n \times \C P^{n-1}\, \vert \,z\in [t] \quad \vert z\vert \le 1\Bigr\}$.
Its boundary is the sphere, which  is the boundary of the ball  in $\C^{n}\simeq\R^{2n}$ which is a Liouville domain.
The basis $B= \C P^{n-1}$ admits  the Morse function $f_{n-1}$ defined earlier, whose critical points have even Morse indices.
\end{exemple}
\begin{exemple}[Tautological complex line bundle over the Grassmannian  $G^+_{2,n}$]~\\
We consider the tautological complex negative  line bundle over the Grassmannian of oriented $2$-planes in $\R^n$ (an oriented real $2$-plane being considered as a complex
$1$-dimensional space).
 The circle bundle, which is the boundary of the corresponding  disk bundle,
 is canonically isomorphic to the unit sphere bundle in the cotangent bundle $T^*S^{n-1}$ to the sphere $S^{n-1}$; indeed
 any element $z$ above an oriented $2$-plane $\pi$ in this circle bundle represents two oriented orthonormal vectors
 $u,v$ in $\R^n$ spanning $\pi$; this can be viewed as an element $u$ in  $S^{n-1}$ and an element $v$ in $T^*_u(S^{n-1})$.
  This unit bundle in the cotangent bundle $T^*S^{n-1}$
 is also the boundary of the Liouville domain defined by all cotangent vectors of length at most $1$ in $T^*S^{n-1}$.
The basis of our tautological complex line bundle is the Grassmannian  $G^+_{2,n}$; when $n$ is even, it admits  the Morse function \cite{Zhou}
$$
f(\pi)=c_1(u^1v^2-u^2v^1)+c_2(u^3v^4-u^4v^3)+\ldots +c_n(u^{2n-1}v^{2n}-u^{2n}v^{2n-1})
$$
with $c_1>c_2>\ldots >c_n>0$, where the oriented $2$-plane $\pi$ is generated by the oriented orthonormal vectors $u=\sum_{i=1}^{2n}u^ie_i$ and $v\sum_{j=1}^{2n}v^je_i$ in $\R^{2n}$.
The critical points of $f$ are the planes spanned by the oriented basis $\{\,e_1,e_2\,\}$, $\{\,e_3,e_4\,\} \ldots \{\,e_{2n-1},e_{2n}\,\}$
and $\{\,e_{2n},e_{2n-1}\,\} ,\ldots \{\,e_4,e_3\,\} , \{\,e_2,e_1\,\} $; they have  Morse indices  $0,2,\ldots,2n-2, 2n-2,\ldots,4n-6, 4n-4$
respectively. These are all even integers.
\end{exemple}

\bibliographystyle{abbrv}
\bibliography{Revision.bib}



\end{document}